\documentclass[graybox]{svmult}

\usepackage{type1cm}        % activate if the above 3 fonts are
\usepackage{makeidx}         % allows index generation
\usepackage{graphicx}        % standard LaTeX graphics tool
\usepackage{multicol}        % used for the two-column index
\usepackage[bottom]{footmisc}% places footnotes at page bottom

\usepackage{newtxtext}       % 
\usepackage[varvw]{newtxmath}       % selects Times Roman as basic font
\usepackage{cleveref}
\usepackage{braket}
\usepackage{enumitem}

\def\Cref{\ref}

\newcommand{\trivial}[2][]{\if\relax\detokenize{#1}\relax{%\hfill\break
      \color{orange} \vspace{0em}$[$#2$]$
      \color{black}
  }\else\ifx#1h\ifcsname showtrivial\endcsname
{%\hfill\break
    \color{orange}\vspace{0em}$[$#2$]$
    \color{black}
}\fi\else {\red Wrong argument!} \fi\fi}

\newcommand{\byhide}[2][]{\if\relax\detokenize{#1}\relax
{\color{orange} \vspace{0em} Plan to delete:  #2}
\else
\ifx#1h\relax\fi
\fi
}

\newcommand{\GL}{\mathrm{GL}}

\newcommand{\AV}{\mathrm{AV}}

\newcommand{\bfone}{\mathbf{1}}

\def\Im{\operatorname{Im}}



\makeatletter
\def\inn#1#2{\left\langle
      \def\ta{#1}\def\tb{#2}
      \ifx\ta\@empty{\;} \else {\ta}\fi ,
      \ifx\tb\@empty{\;} \else {\tb}\fi
      \right\rangle}
\def\binn#1#2{\left\lAngle
      \def\ta{#1}\def\tb{#2}
      \ifx\ta\@empty{\;} \else {\ta}\fi ,
      \ifx\tb\@empty{\;} \else {\tb}\fi
      \right\rAngle}
\makeatother

\makeatletter
\def\binn#1#2{\overline{\inn{#1}{#2}}}
\makeatother

\usepackage{tikz}
\usetikzlibrary{matrix,arrows,positioning,cd,backgrounds}
\usetikzlibrary{decorations.pathmorphing,decorations.pathreplacing}

\usepackage{upgreek}

\long\def\delete#1{}

\newcommand{\BC}{{\mathbb {C}}}

\newcommand{\CO}{{\mathcal {O}}}

\newcommand{\RU}{{\mathrm {U}}}

\def\I{{\mathrm{I}}}

\newcommand{\sgn}{\operatorname{sgn}}

\newcommand{\oU}{\operatorname{U}}

\newcommand{\g}{\mathfrak g}

\renewcommand{\k}{\mathfrak k}
\newcommand{\h}{\mathfrak h}

\renewcommand{\u}{\mathfrak u}

\renewcommand{\l}{\mathfrak l}
\renewcommand{\t}{\mathfrak t}
\newcommand{\s}{\mathfrak s}

\newcommand{\Z}{\mathbb{Z}}
\def\C{\mathbb{C}}
\newcommand{\R}{\mathbb R}

\newcommand{\la}{\langle}
\newcommand{\ra}{\rangle}

\newcommand{\be}{\begin {equation}}
\newcommand{\ee}{\end {equation}}

\def\flushl#1{\ifmmode\makebox[0pt][l]{${#1}$}\else\makebox[0pt][l]{#1}\fi}
\def\flushr#1{\ifmmode\makebox[0pt][r]{${#1}$}\else\makebox[0pt][r]{#1}\fi}

\def\half{{\tfrac{1}{2}}}

\def\sspan{\mathrm{span}}

\def\Irr{\mathrm{Irr}}

\def\ckG{\check{G}}

\def\PP{\mathsf{PAP}}

\def\bsgn{\overline{\sgn}}

\def\Im{\mathrm{Im}}

\def\Coh{\mathrm{Coh}}

\def\yrow#1{\left[#1\right]_{\mathrm{row}}}
\def\ycol#1{\left[#1\right]_{\mathrm{col}}}

\def\Nil{{\mathrm{Nil}}}
\def\Sp{\mathrm{Sp}}
\def\Spin{\mathrm{Spin}}

\def\Unip{\mathrm{Unip}}
\def\cKgen{\cK^{\mathrm{gen}}}

\def\Unipgen{{\Unip^\mathrm{gen}_{\check \CO}}}

\def\cC{\mathcal{C}}
\def\cO{\mathcal{O}}
\def\cCb{\cC^{\mathrm b}}
\def\cCg{\cC^{\mathrm g}}

\def\bfrr{{\mathbf r}}

\def\sfH{\mathsf{H}}
\def\sfS{\mathsf{S}}
\def\sfW{\mathsf{W}}
\def\Ind{\mathrm{Ind}}
\def\ckcO{{\check{\mathcal{O}}}}
\def\cK{\mathcal{K}}
\def\bC{\mathbb{C}}
\def\fhh{\mathfrak{h}}
\def\fgg{\mathfrak{g}}
\def\Lie{\mathrm{Lie}}

\def\SO{\mathrm{SO}}

\def\abs#1{\left|{#1}\right|}

\def\bfoo{{\mathbf o}}

\def\wtG{{\widetilde G}}

\def\bC{{\mathbb C}}
\def\bZ{{\mathbb Z}}

\def\bN{{\mathbb N}}
\def\igg{\sqrt{-1} \mathfrak g^*}
\def\fnn{{\mathfrak n}}
\def\wtG{\widetilde{G}}

\def\halfone{\half\bfone}
\def\cKG{\cK(\wtG)}
\def\cKlamG{\cK_\lambda(\wtG)}
\def\cKgenG{\cK^{\mathrm{gen}}(\wtG)}

\def\cKclsG{\cK^{\mathrm{cls}}(\wtG)}

\def\ckalpha{{\check\alpha}}

\def\fsl{{\mathfrak{sl}}}

\def\ftt{{\mathfrak t}}
\def\fuu{{\mathfrak u}}
\def\fkk{{\mathfrak k}}
\def\fll{{\mathfrak l}}
\def\fss{{\mathfrak s}}
\def\fqq{{\mathfrak q}}
\def\ree{{\mathrm e}}
\def\rff{{\mathrm f}}
\def\rhh{{\mathrm h}}

\def\baroop{\overline{\bfoo'}}

\begin{document}

\title*{{Genuine special unipotent representations of spin groups}}
%\author{Name of First Author and Name of Second Author}
\author{Dan Barbasch, Jia-Jun Ma, Binyong Sun, and Chen-Bo Zhu}
\institute{Dan Barbasch \at Department of Mathematics, 310 Malott Hall, Cornell University, Ithaca, New York 14853, USA \email{dmb14@cornell.edu}
\and Jia-Jun Ma \at School of Mathematical Sciences, Xiamen University, Xiamen, 361005, China\\
Department of Mathematics, Xiamen University Malaysia campus, Sepang, Selangor Darul Ehsan, 43900,  Malaysia
\email{hoxide@xmu.edu.cn}
\and Binyong Sun\at Institute for Advanced Study in Mathematics \& New Cornerstone Science Laboratory, Zhejiang University,  Hangzhou, 310058, China \email{sunbinyong@zju.edu.cn}
\and Chen-Bo Zhu \at Department of Mathematics, National University of Singapore, 10 Lower Kent Ridge Road, Singapore 119076 \email{matzhucb@nus.edu.sg}
}

\maketitle

\hspace*{\fill}\emph{To Toshiyuki Kobayashi with friendship and admiration}\hspace*{\fill}

\vskip 20pt 

\abstract{We determine all genuine special unipotent representations of real spin groups and quaternionic spin groups, and show in particular that all of them are unitarizable. We also show that there are no genuine special unipotent representations of complex spin groups. 
} 

\section{Introduction and the main results}\label{sec:intro}

In two earlier works  \cite{BMSZ1,BMSZ2}, the authors construct and classify special unipotent representations of real classical groups (in the sense of Arthur and Barbasch-Vogan; the terminology comes from \cite{Lu}). As a direct consequence of the construction and the classification, the authors show that all special unipotent representations of real classical groups are unitarizable, as predicted by the Arthur-Barbasch-Vogan conjecture (\cite[Section 4]{ArUni}, \cite[Introduction]{ABV}). For quasi-split classical groups, the unitarity is independently established in \cite{AAM,AM}. 

In this paper we first consider a real or quaternionic spin group $G$ and determine all genuine special unipotent representations of $G$. (The analogous but simpler case of the complex spin group is briefly discussed at the end of this section.) In particular, we show that all of them are unitarizable. The paper may be considered as a companion paper of \cite{BMSZ1,BMSZ2}. 
The general strategy follows those of \cite{BMSZ1,BMSZ2} and it consists of two steps. First we use the coherent continuation representation to count the number of special unipotent representations (attached to a nilpotent $\check G$-orbit $\check \CO$ in $\check \g$; notations to follow), as in \cite{BMSZ1}. Second we construct the exact number of special unipotent representations to match the counting and arrive at the classification as a result. This second task is made significantly easier in the current case of spin groups, due to the fact that there are only a very small number of genuine special unipotent representations. In the case of real classical groups, both tasks of the construction and distinguishing the constructed representations pose significant challenges, which we overcome in \cite{BMSZ2} by Howe's method of theta lifting \cite{Howe79,Howe89} and Vogan's theory of associated cycles \cite{Vo89}.  

\medskip
\medskip
Let $G_{\bC} = \Spin(m,\bC)$ ($m\geq 2$) be the complex spin group. Its rank is $n:=\lfloor \frac{m}{2}\rfloor$. Let $G$ be a real form of $G_\bC$, namely $G$ is the fixed point group of an involutive anti-holomorphic automorphism $\sigma$ of $G_\bC$. Up to conjugation by $G_\bC$, every real form of $G_\bC$ is a real spin group or a quaternionic spin group. We may thus assume without loss of generality  that  $G$ is the real spin group $\Spin(p,q)$ ($p+q=m$), or the quaternionic spin group $\Spin^*(2n)$ (when $m$ is even).  We will be concerned with special unipotent representations of $G$ (in the sense of Arthur and Barbasch-Vogan; see \cite{BVUni,ABV}). 

It will be convenient to introduce a slightly larger group $\widetilde G$. Consider the natural exact sequence 
\[
 \{1\}\rightarrow \{\pm 1\}\rightarrow G_\bC\rightarrow \SO(m,\bC)\rightarrow \{1\}. 
\]
The automorphism $\sigma$ descends to an 
involutive anti-holomorphic automorphism  of $\SO(m,\bC)$, to be denoted by $\sigma_0$. Let $G_0$ denote the fixed point group of $\sigma_0$, and write $\widetilde G$ for the preimage of $G_0$ under the covering homomorphism $G_\bC\rightarrow \SO(m,\bC)$. Then $G_0$ is a real special orthogonal group $\SO(p,q)$ ($p+q=m$), or a quaternionic orthogonal group $\SO^*(2n)$ (when $m$ is even). We summarize with an exact sequence
\[
 \{1\}\rightarrow \{\pm 1\}\rightarrow \widetilde G\rightarrow G_0\rightarrow \{1\}, 
\]
where $G_0=\SO(p,q)$ or $\SO^*(2n)$.  

If $G=\Spin(p,q)$ and $pq=0$, or $G=\Spin^*(2n)$, then $\widetilde G=G$. Otherwise, $\widetilde G$ contains $G$ as a subgroup of index $2$.

We first consider representations of $\widetilde G$ instead of $G$, and we then relate the representation theory of $G$ to that of $\widetilde G$ by Clifford theory.

We will work in the category of Casselman-Wallach representations \cite[Chapter 11]{Wa2}. An irreducible Casselman-Wallach representation of $\widetilde G$ either descends to a representation of $G_0$ or is genuine in the sense that the central subgroup $\{\pm 1\}\subset \widetilde G$ (the kernel of the covering homomorphism $\widetilde G \rightarrow G_0$) acts via the non-trivial character. (The notion of ``genuine" will be used in some similar situations without further explanation.) Since special unipotent representations of $G_0$ have been classified \cite{BMSZ1,BMSZ2}, we will focus on genuine special unipotent representations of $\widetilde G$.

The Langlands dual of $G$ is 
\begin{equation}\label{eq:Langlands}
 \check G:=\begin{cases}
    \SO(m,\bC)/\{\pm 1\}, &  \text{if $m$ is even}; \\
    \mathrm{Sp}(m-1,\bC)/\{\pm 1\}, &  \text{if $m$ is odd}. 
  \end{cases}
\end{equation}
Denote by $\check \g$ the Lie algebra of $\check G$. 

Let $\check \CO$ be a nilpotent $\check G$-orbit in $\check \g$ (see \cite{CM} for basic facts about nilpotent orbits). Denote by  $\Unip^{\mathrm{gen}}_\ckcO(\widetilde G)$ the set of isomorphism classes of genuine special unipotent representations of $\widetilde G$ attached to $\check \CO$. We refer the reader to \cite[Section 2]{BMSZ1} for a comprehensive discussion of special unipotent representations in the case of real classical groups. 

\delete{ 
 \begin{equation}\label{eq:defuni}
        \Unip^{\mathrm{gen}}_{\check \CO}(\widetilde G):= 
       \{\pi\in \Irr(\widetilde G)\mid \pi \textrm{ is genuine  and annihilated by } I_{\check \CO}\}.
\end{equation}
}

Denote by $\g$ and $\g_\C$ the Lie algebras of $G$ and $G_\bC$, respectively. Let $\CO\subset \g_\bC^*$ denote the Barbasch-Vogan dual of $\check \CO$ (see \cite[Appendix]{BVUni} and \cite[Section 1]{BMSZ0}). 

Our first result will count the set  $\Unipgen(\widetilde G)$. 
%We have the  following counting result for $\widetilde G$. 
As usual we will not distinguish between a nilpotent orbit $\check \CO$ and its Young diagram  (when there is no confusion). Recall that $\check \CO$ is called very even if all its row lengths are even. 

\begin{theorem}\label{thm:Spinpq0}
(a) If some row length of $\check \CO$ occurs with odd multiplicity, then the set $\Unipgen(\widetilde G)$ is empty. 

\noindent (b) Suppose that all row lengths of $\check \CO$ occur with even multiplicity. Then 
%the cardinality of the set $\Unipgen(\widetilde G)$ is as follows: 
\begin{eqnarray*}
&&\sharp(\Unipgen(\widetilde G))\\
&=&
\begin{cases}
    1, &  \text{if $G=\Spin(p,q)$, $\abs{p-q}=1$}; \\
    1, &  \text{if $G=\Spin(p,q)$, $p=q$ and $\check \CO$ is very even}; \\
     2, &  \text{if $G=\Spin(p,q)$, $p=q$ and $\check \CO$ is not very even}; \\
      \sharp (\widetilde G\backslash \sqrt{-1} \g^* \cap \CO), &  \text{if $G=\Spin^*(2n)$ and $\check \CO$ is very even}; \\
       0, &  \text{otherwise}.
  \end{cases}
\end{eqnarray*}
\end{theorem}

Here and henceforth, $\sharp$ indicates the cardinality of a finite set.

%\begin{remark}\label{rmk:relevent}
Note that when $m$ is even and $\check \CO$ is very even, there are actually two nilpotent $\check G$-orbits in $\check \g$ that have the same Young diagram as $\check \CO$. The two orbits are given labels $I/II$ as $\check \CO_I$ and $\check \CO_{II}$. Write $\CO_{I}$ and $\CO_{II}$  respectively for their Barbasch-Vogan duals. If $G=\Spin^*(2n)$ and $\check \CO$ is very even, then among  the sets $\sqrt{-1} \g^* \cap \CO_I$ and $\sqrt{-1} \g^* \cap \CO_{II}$, one is empty and the other is not. We will use the convention for the labels so that 
\begin{equation}\label{def:relevant}
\sqrt{-1}\g^*\cap \cO_I\neq \emptyset
\end{equation}
and we say $\ckcO_I$ is $G$-relevant. We refer the reader to \cite[Section 2.6]{BMSZ1} for the  notion of $G$-relevance in a more general context. 

%Take a Siegel parabolic $P$ in $G$. Let $\ckP$ be the Siegel parabolic in $\ckG$ corresponding to the complexification of $P$ and $\ckL$ be its Levi subgroup. The group $\ckL$ is well defined up to $\ckG$-conjugation.  Then the set of all relevant very even orbits  can be characterized by the condition $\ckcO_I\cap \Lie(\ckL)\neq \emptyset$.  
\trivial[h]{
Let $B$ be the minimal parabolic subgroup of $G$, then $\cO_I$ is characterized by $\cO_i \cap \Lie(B)\neq \emptyset$. (Since $B$ are unique up to $G$-conjugate, we see that the condition is independent of the choice of $B$.  
}

Let $\widetilde{\mathrm{sgn}}$ denote the quadratic character on $\widetilde G$ whose kernel equals $G$. 

\begin{theorem}\label{thm:Spinpq03}
Let $\pi\in \Unipgen(\widetilde G)$. If $G=\Spin(p,q)$ and $p=q$ is odd, then $\pi\otimes \widetilde{\mathrm{sgn}}$ is not isomorphic to $\pi$. In all other cases, $\pi\otimes \widetilde{\mathrm{sgn}}$ is  isomorphic to $\pi$. 
\end{theorem}

Similar to $\Unipgen(\widetilde G)$, define the set $\Unipgen(G)$ of isomorphism classes of genuine special unipotent representations of $G$ attached to $\check \CO$. 
By Clifford theory and Theorem \ref{thm:Spinpq03}, Theorem \ref{thm:Spinpq0} implies the  following counting result for $G$. 

\begin{theorem}\label{thm:Spinpq4}
(a) If some row length of $\check \CO$ occurs with odd multiplicity, then the set $\Unipgen(G)$ is empty. 

\noindent (b) Suppose that all row lengths of $\check \CO$ occur with even multiplicity. Then 
%the cardinality of the set $\Unipgen(G)$ is as follows: 
\begin{eqnarray*}
&&\sharp(\Unipgen(G))\\
&=&
\begin{cases}
    2, &  \text{if $G=\Spin(p,q)$, $\abs{p-q}=1$}; \\
    2, &  \text{if $G=\Spin(p,q)$, $p=q$ is even and $\check \CO$ is very even}; \\
     4, &  \text{if $G=\Spin(p,q)$, $p=q$ is even and $\check \CO$ is not very even}; \\
      1, &  \text{if $G=\Spin(p,q)$, $p=q$ is odd}; \\
      \sharp ( G\backslash \sqrt{-1} \g^* \cap \CO), &  \text{if $G=\Spin^*(2n)$ and $\check \CO$ is very even}; \\
       0, &  \text{otherwise}.
  \end{cases}
\end{eqnarray*}
\end{theorem}

We aim to describe all representations in $\Unipgen(\widetilde G)$ and $\Unipgen(G)$. 
In the rest of this introduction, we assume  that the set $\Unipgen(\widetilde G)$ is non-empty, or equivalently, the set $\Unipgen(G)$ is non-empty. Theorem \ref{thm:Spinpq0} in particular implies that all row lengths of $\check \CO$ occur with even multiplicity. 
Write 
\[
r_1= r_1\geq  r_2= r_2\geq  \dots\geq  r_k=r_k>0\quad  (k\geq 1)
\]
for the row lengths of $\check \CO$. 

As usual, let $\mathbb H$ denote the division algebra of Hamilton's quaternions. We omit the proof of the following elementary lemma. 
\begin{lemma}\label{lemf}
Up to conjugation by $G_0$, there is a unique Levi subgroup $L_0$ of $G_0$ with the following properties:
\begin{enumerate}[label=(\alph*)]
\item 
\[
L_0\cong \begin{cases}
    \mathrm{GL}_{r_1}(\R)\times \mathrm{GL}_{r_2}(\R)\times \dots \times \mathrm{GL}_{r_k}(\R), &  \text{if $G_0=\SO(p,q)$};\\
    \mathrm{GL}_{\frac{r_1}{2}}(\mathbb H)\times \mathrm{GL}_{\frac{r_2}{2}}(\mathbb H)\times \dots \times \mathrm{GL}_{\frac{r_k}{2}}(\mathbb H), &  \text{if $G_0=\SO^*(2n)$}; 
  \end{cases}
\]
\item the orbit $\CO$ equals the induction of the zero orbit from $\l_\bC$ to $\g_\bC$, where $\l_\bC$ denotes the complexified Lie algebra of $L_0$. 
\end{enumerate}

\end{lemma}

We remark that the second condition in the above lemma  is implied by the first one, unless $G_0=\SO(p,q)$, $p=q$, and $\check \CO$ is very even. 

Let $L_0$ be as in Lemma \ref{lemf} and write  $\widetilde L$ for the preimage of $L_0$ under the covering homomorphism $G_\bC\rightarrow \SO(m,\bC)$. 
Write $\mathrm N_{\widetilde G}(\widetilde L)$ for the normalizer of $\widetilde L$ in $\widetilde G$.

\begin{lemma}\label{lemf2}
(a) Suppose that $G_0=\SO(p,q)$. If either $\abs{p-q}=1$, or $p=q$ and $\check \CO$ is very even, then up to conjugation by $\mathrm N_{\widetilde G}(\widetilde L)$, there exists a unique genuine character on $\widetilde L$ of finite order. 

\noindent (b) Suppose that $G_0=\SO(p,q)$. If $p=q$ and $\check \CO$ is not very even, then up to conjugation by $\mathrm N_{\widetilde G}(\widetilde L)$, there are precisely two   genuine characters on $\widetilde L$ of finite order.

\noindent (c) Suppose that $G_0=\SO^*(2n)$. Then the covering homomorphism $\widetilde L\rightarrow L_0$ uniquely splits, and $\widetilde L$ has a   unique genuine character of finite order. 
 
\end{lemma}
\begin{proof}
  Suppose that $G_0=\SO(p,q)$ and $p=q$. Let $Z$ denote the inverse image of $\{\pm 1\}$ under the covering homomorphism $\widetilde G\rightarrow \SO(p,p)$. Then $Z$ is a central subgroup of $\widetilde G$ which is contained in $\widetilde L$. Moreover, 
\[
  Z\cong \begin{cases}
    \Z/2\Z\times \Z/2\Z, &  \text{if $p$ is even}; \\
   \Z/4\Z, &  \text{if $p$ is odd}. 
  \end{cases}
\]
When $\check \CO$ is not very even, it is easy to see that there are two genuine characters on $\widetilde L$ of finite order that have distinct restrictions to $Z$. Therefore up to conjugation by $\mathrm N_{\widetilde G}(\widetilde L)$, there are at least two   genuine characters on $\widetilde L$ of finite order. The rest of the lemma is easy to prove and we omit the details. 
\end{proof}

Let $\widetilde P$ be a parabolic subgroup of $\widetilde G$ containing $\widetilde L$ as a Levi factor. 
Let $\chi$ be a genuine character on $\widetilde L$ of finite order, to be viewed as a character of $\widetilde P$ as usual. Put
\[
I(\chi):=\Ind_{\widetilde P}^{\widetilde G} \chi \qquad (\textrm{normalized smooth parabolic induction}).
\]
This is called a degenerate principal series representation. 

\begin{theorem}\label{unip0}
   (a) Suppose that $G=\Spin(p,q)$. Then $I(\chi)$ is irreducible and belongs to $\Unipgen(\widetilde G)$. Moreover
  \[
  \Unipgen(\widetilde G)=\begin{cases}
    \{I(\chi)\}, &  \text{if either $\abs{p-q}=1$, or $p=q$ and $\check \CO$ is very even}; \\
   \{I(\chi_1), I( \chi_2)\}, &  \text{if $p=q$ and $\check \CO$ is not very even}. 
  \end{cases}
  \] 
  Here $\chi_1$ and $\chi_2$ are two genuine characters on $\widetilde L$ of finite order that are not conjugate by $\mathrm N_{\widetilde G}(\widetilde L)$.  

\noindent (b) Suppose that $G=\Spin^*(2n)$. Then for every $\mathbf o\in \widetilde G\backslash \sqrt{-1}\g^* \cap \CO$, there is a unique (up to isomorphism) irreducible Casselman-Wallach representation $\pi_\mathbf o$ of $\widetilde G$ that occurs in $I(\chi)$ whose wave front set equals the closure of $\mathbf o$. Moreover, \[
\Unipgen(\widetilde G)=\{\pi_\mathbf o\mid \mathbf o\in \widetilde G\backslash  \sqrt{-1}\g^* \cap \CO\}
\]
and
%\begin{equation}\label{decq}
\[  I(\chi)\cong \bigoplus_{\mathbf o\in \widetilde G\backslash \sqrt{-1}\g^* \cap \CO} \pi_\mathbf o.
\]
%\end{equation}

\end{theorem}

Now we return to consider the group $G$. Assume that $G=\Spin(p,q)$. \trivial[h]{
Since $\abs{p-q}\leq 1$ and $m=p+q\geq 2$, we know that $pq\neq 0$. Thus $\widetilde G$ contains $G$ as a subgroup of index two.
Pick an arbitrary element $g_0\in \widetilde G \setminus G$. For every Casselman-Wallach representation $\pi$ of $G$, write $\pi^{g_0}$ for its twist by $g_0$ (via conjugation). In other words, $\pi^{g_0}$ is the Casselman-Wallach representation of $G$ whose underlying space is the same as that of $\pi$ (to be denoted by $V_\pi$), and whose action is given by the composition of 
\[
 G\xrightarrow{g\mapsto g_0 g g_0^{-1}} G\xrightarrow{\textrm{the representation $\pi$}}\mathrm{GL}(V_\pi). 
\]
}
Put $P:=\widetilde P\cap G$ and $L=\widetilde L\cap G$. Then we have the  identification 
\[
I(\chi)|_G=\Ind_P^{G} \chi|_L
\]
of representations of $G$.
By Clifford theory and Theorem \ref{thm:Spinpq03}, if $p=q$ is odd, then $I(\chi)|_G$ is irreducible. In all other cases, $I(\chi)|_G$ is the direct sum of two irreducible subrepresentations that are not isomorphic to each other.  Write  $I_1(\chi)$ and $I_2(\chi)$ for these two irreducible subrepresentations.

By Clifford theory, part (a) of Theorem \ref{unip0} easily implies the following result. We omit the details.   

\begin{theorem}\label{spinc}
    Assume that $G=\Spin(p,q)$. Then 
    \begin{eqnarray*}
&&\Unipgen(G)\\
&=&
\begin{cases}
    \{I_1(\chi), I_2(\chi) \}, &  \text{if $\abs{p-q}=1$}; \\
     \{I_1(\chi), I_2(\chi) \}, &  \text{if $p=q$ is even},\\
    &\quad  \text{and $\check \CO$ is  very even};  \\
      \{I_1(\chi_1), I_2(\chi_1), I_1(\chi_2), I_2(\chi_2)\}, &  \text{if $p=q$ is even},\\
    &\quad  \text{and $\check \CO$ is not very even}; \\
      \{I(\chi)|_G\}, &  \text{if $p=q$ is odd}.
  \end{cases}
\end{eqnarray*}
Here $\chi_1$ and $\chi_2$ are two genuine characters on $\widetilde L$ of finite order that are not conjugate by  $\mathrm N_{\widetilde G}(\widetilde L)$. %the normalizer of $\widetilde L$ in $\widetilde G$. 
\end{theorem}

\trivial[h]{
By Clifford theory, part (a) of Theorem \ref{unip0} easily implies the following result. We omit the details.   

\begin{theorem}\label{spinc}
    Assume that $G=\Spin(p,q)$. Then the representation $I(\chi|_L)$ is irreducible and belongs to $\Unipgen(G)$. Moreover, 
    \begin{eqnarray*}
&&\Unipgen(G)\\
&=&
\begin{cases}
    \{I(\chi|_L), I(\chi|_L)^{g_0} \}, &  \text{if $\abs{p-q}=1$}; \\
     \{I(\chi|_L), I(\chi|_L)^{g_0} \}, &  \text{if $p=q$}, \textit{ and $\check \CO$ is very even}; \\
      \{I({\chi_1}|_L), I({\chi_1}|_L)^{g_0},I({\chi_2}|_L), I({\chi_2}|_L)^{g_0} \}, &  \text{if $p=q$ is even},\\
    &\quad  \text{and $\check \CO$ is not very even}; \\
      \{I({\chi}|_L)\}, &  \text{if $p=q$ is odd},\\
      &\quad \text{and $\check \CO$ is not very even}.
  \end{cases}
\end{eqnarray*}
Here $\chi_1$ and $\chi_2$ are two genuine characters on $\widetilde L$ of finite order that are not conjugate by  $\mathrm N_{\widetilde G}(\widetilde L)$. %the normalizer of $\widetilde L$ in $\widetilde G$. 
\end{theorem}
}
The following unitarity result is an obvious consequence of  Theorems \ref{unip0} and \ref{spinc}.

\begin{theorem}\label{genuine}
All genuine special unipotent representations of $\widetilde G$ or $G$ are unitarizable. 
\end{theorem}

\begin{remark}
Theorem \ref{genuine}, together with the result of \cite{BMSZ2} on special unipotent representations of $G_0$, implies that all special unipotent representations of $\widetilde G$ or $G$ are unitarizable.
\end{remark}

We now examine the case of the complex spin group. Recall that  $G_\C=\Spin (m,\C)$.
View $G_\C$ as a real reductive group. Then its  Langlands dual is $\ckG\times \ckG$, where $\ckG$ is given in \eqref{eq:Langlands}. 
As before, denote by  $\Unip^{\mathrm{gen}}_{\ckcO_\C}(G_\C)$ the set of isomorphism classes of genuine special unipotent representations of $G_\C$ attached to $\check \CO_\C$, where $\check \CO_\C$ is a nilpotent $\check G\times \check G$-orbit in $\check \g \times \check \g$ (in the obvious notation). 

The following theorem should be known within the expert community. As the proof is along the same line as that of Theorem \ref{thm:Spinpq0}, we will state the result without proof.  

\begin{theorem} The set $\Unip^{\mathrm{gen}}_{\ckcO_\C}(G_\C)$ is empty, namely there exists no genuine special unipotent representation of $G_\C=\Spin (m,\C)$. 
%. Let $\ckcO$ be a nilpotent orbit of $\ckG$. The pullback map induces a bijection
% \[
%\Unip_{\ckcO\times \ckcO}(\SO(m,\bC))\rightarrow %\Unip_{\ckcO\times \ckcO}(\Spin(m,\bC)).  
%\]
\end{theorem}

We remark that Losev, Mason-Brown and Matvieievskyi \cite{LMBM} have recently proposed a notion of unipotent representations for a complex reductive group, extending the notion of special unipotent representations. See also \cite[Section 2.3]{B17} for a different version of the notion proposed by Barbasch earlier. Genuine unipotent representations of $\Spin(m,\bC)$ do exist and they play a key role in the determination of the unitary dual of $\Spin(m,\bC)$ (see for example, \cite{Brega} and \cite{WZ}).

\section{Proof of Theorem \ref{thm:Spinpq0}}
\label{secpf1}
In this section, we apply the method developed in \cite{BMSZ1} to prove Theorem~\ref{thm:Spinpq0}. We will also follow \cite{BMSZ1} in our choice of notations and conventions.  

\subsection{Coherent continuation representation of a spin group}
We adopt the formulation of coherent continuation representation in \cite[Sections 3 and 4] {BMSZ1}. The original references include \cite{Sch, Zu, SpVo, Vg}. 

We make the following identifications: 
\begin{itemize}
    \item  The dual $\fhh^*$ of the abstract Cartan subalgebra $\fhh$ of $\g_\bC$ is identified with $\bC^n$. 
\item 
The analytic weight lattice $Q\subset \fhh^*$ of $G_{\bC}$ is identified with $\bZ^{n}\sqcup (\half \bfone +\bZ^n)$. 
Here $\bfone := (1, 1, \cdots, 1)\in \bZ^n$. 
\item The analytic weight lattice $Q_0$ of $\SO(m,\bC)$ 
is identified with the sublattice $\bZ^n$ of $Q$ via the natural quotient map $G_\bC \rightarrow \SO(m,\bC)$. 
%\item The image of $G$ in $\SO(m,\bC)$ is denoted by $\barG$. 
\end{itemize}

We adopt the following notations:
\begin{itemize}
\item 
Let $\cKG$ be the 
Grothendieck group (with coefficients in $\bC$) of the category of Casselman-Wallach representations of $\wtG$. 
\item Let $\cKlamG$ be the subgroup of $\cKG$ generated by irreducible representations of $\wtG$ with infinitesimal character
$\lambda\in \fhh^{*}$,
\item 
Let $\cKgenG$ be the subgroup of $\cKG$ generated by irreducible genuine representations of $\wtG$. 
\item  Let $\cKclsG$ be the subgroup of $\cKG$ generated by irreducible representations of $\wtG$ which factor through $G_0$.
We identify $\cKclsG$ with the Grothendick group $\cK(G_0)$ of the category of Casselman-Wallach representations of $G_0$. 
%\item $\cKgenlamG := \cKlamG\cap \cKgenG$ and $\cKclslamG := \cKlamG\cap \cKclsG$
\end{itemize}

Let $W$ be abstract Weyl group of $G_\bC$, which acts on $\fhh^*$ in the standard way. 

For the rest of this section, let $\Lambda \subset \fhh^{*}$ be a $Q$-coset.  
Denote by $W_{\Lambda}$
%\[
%W_\Lambda := \set{w\in W | w\cdot \Lambda =\Lambda}
%\]
the stabilizer of $\Lambda$ in $W$. Also we have the integral Weyl group 
\begin{equation}\label{intwgp}
W(\Lambda) := \set{w\in W | w\lambda -\lambda\textrm{ is in the root lattice for every } \lambda \in \Lambda}, 
\end{equation}
%\[
%W(\Lambda) := \set{w\in W | \inn{w\lambda -\lambda}{\check \alpha}\in \bZ, \forall \lambda \in \Lambda, \alpha\in \Delta}, 
%\]
which is a subgroup of $W_\Lambda$. %Here $\Delta\subset \fhh^*$ denotes the root system of $\g$, and $\check \alpha\in \fhh$ denotes the coroot corresponding to the root $\alpha$.

\begin{definition} \label{cohwtG}
The space $\Coh_\Lambda (\cKG)$ of $\cKG$-valued coherent families based on $\Lambda$ is the vector space of all   
maps $\Psi\colon \Lambda \to \cKG$ such that, for all $\lambda \in \Lambda$,  
\begin{itemize}
    \item $\Psi(\lambda)\in \cKlamG$, and  
    \item for any holomorphic finite-dimensional representation $F$ of $G_\bC$,
    \[
    F\otimes \Psi(\lambda)  = \sum_{\mu} \Psi(\lambda+\mu), 
    \]
    where $\mu$ runs over the set of all weights (counting multiplicities) of $F$.  
\end{itemize}
\end{definition}

The space $\Coh_\Lambda(\cKG)$ is a $W_\Lambda$-module under the action
\[(w\cdot \Psi)(\lambda) = \Psi(w^{-1}\lambda), \qquad w\in W_{\Lambda}.\] 
This is called the coherent continuation representation.

%Since $X$ has index $2$ in $Q$, we have  
Fix a decomposition 
\[
\Lambda = \Lambda_{1} \sqcup \Lambda_{2}
\]
where $\Lambda_{1}$ and $\Lambda_{2}$ are $Q_0$-cosets. Then $\Lambda_1 = \halfone+\Lambda_{2}$. 
%such that where $\Lambda_{1}$ and $\Lambda_{2}$ are two different points in  $\fhh^{*}/X$ such that $\Lambda_1 = \halfone+\Lambda_{2}$.
Similar to  \eqref{intwgp}, we  define the integral Weyl groups $W(\Lambda_1)$ and $ W(\Lambda_2)$. Then  
\[W(\Lambda) = W(\Lambda_1) = W(\Lambda_2).\]
\trivial[h]{
Note that $w\cdot \halfone - \halfone\in \bZ$.  
$\inn{w\cdot \lambda -\lambda}{\ckalpha}\in \bZ \longleftrightarrow
\inn{w\cdot (\lambda +\halfone) -(\lambda+\halfone)}{\ckalpha}\in \bZ$.
Warning: $W_\Lambda \neq W_{\Lambda_1}$ in general. For example, when $\frac{1}{4}\bfone \in \Lambda_1$. 
} 

Let $\Coh_{\Lambda_1}(\cKclsG) := \Coh_{\Lambda_1}(\cK(G_0))$ be the coherent continuation representation defined for $G_0$ based on $\Lambda_1$
(replace  $\cKG$ by $\cK(G_0)$, $\Lambda$ by $\Lambda_1$ and $G_\bC$ by $\SO(m,\bC)$ in Definition~\ref{cohwtG}). 
Similarly, let $\Coh_{\Lambda_2}(\cKgenG)$ be the coherent 
continuation representation defined for $\wtG$ based on $\Lambda_2$  
(replace $\cKG$ by $\cKgenG$, $\Lambda$ by $\Lambda_2$ and $G_\bC$ by $\SO(m,\bC)$ in Definition~\ref{cohwtG}).

Define a subspace of $\Coh_{\Lambda}(\wtG)$ by  
\[
  \Coh_{\Lambda_{1},\Lambda_{2}}(\wtG):= \Set{ \Psi
    \in \Coh_{\Lambda}(\wtG) | 
  %\colon \Lambda \rightarrow \cKG|
    \begin{array}{l}
     % \forall \lambda \in \Lambda, \Psi(\lambda)\in \cK_{\lambda}(G),\\
      \Im(\Psi|_{\Lambda_{1}})\subset \cKclsG,\\
      \Im(\Psi|_{\Lambda_{2}})\subset \cKgenG
    \end{array}
  }.
\]
\delete{
\[
  \Coh_{\Lambda_{1},\Lambda_{2}}(\wtG):= \Set{ \Psi
    \in \Coh_{\Lambda}(\wtG) | 
  %\colon \Lambda \rightarrow \cKG|
    \begin{array}{l}
     % \forall \lambda \in \Lambda, \Psi(\lambda)\in \cK_{\lambda}(G),\\
      \Im(\Psi|_{\Lambda_{1}})\subset \cKclsG,\\
      \Im(\Psi|_{\Lambda_{2}})\subset \cKgenG
    \end{array}. 
  }
\]
}

The following lemma reduces the counting of genuine representations of $\wtG$ to that of $G_0$-representations. 
\begin{lemma}\label{lem:iso}
  Restrictions induce the following isomorphisms of $W(\Lambda)$-modules:
  \[
    \Coh_{\Lambda_{1}}(\cK(G_0)) \xleftarrow{ \ \ \ \cong \ \ } \Coh_{\Lambda_{1},\Lambda_{2}}(\cKG) \xrightarrow{ \ \ \cong \ \ \ } \Coh_{\Lambda_{2}}(\cKgenG).
  \]
  Moreover, 
  \begin{equation}\label{eq:coheq}
  \Coh_{\Lambda}(\cKG) =  \Coh_{\Lambda_1,\Lambda_2}(\cKG) \oplus \Coh_{\Lambda_2,\Lambda_1}(\cKG). 
  \end{equation}
\end{lemma}
\begin{proof}
  The $W(\Lambda)$-equivariance of the maps is clear. 
  Let $\Psi$ be in $ \Coh_\Lambda(\cKG)$.
  Since $\Psi|_{\Lambda_1}$ and $\Psi|_{\Lambda_2}$ are related by tensoring with genuine finite dimensional representations,  %.  in \cite[Section~7]{Vg}, 
 we see that the following conditions are equivalent: 
 \begin{itemize}
     \item  $\Psi\in \Coh_{\Lambda_1,\Lambda_2}(\cKG)$,
    \item    $\Psi(\lambda_1)\in \cKclsG$ for some regular element $\lambda_1\in \Lambda_1$,
    \item    $\Psi(\lambda_1)\in \cKclsG$ for all elements $\lambda_1\in \Lambda_1$,
    \item    $\Psi(\lambda_2)\in \cKgenG$ for some regular element $\lambda_2\in \Lambda_2$,
    \item    $\Psi(\lambda_2)\in \cKgenG$ for all elements $\lambda_2\in \Lambda_2$,
 \end{itemize}
  The maps are isomorphisms since the evaluation $\Psi \mapsto \Psi(\lambda)$ induces an isomorphism for any regular element $\lambda$ in $\Lambda$. See \cite[Section~7]{Vg}. 
\end{proof}

\def\LCO{{}^L\mathscr C_{\ckcO}}
The following counting result on genuine special unipotent representations will follow from Lemma~\ref{lem:iso}. Recall that attached to $\ckcO$, there is an infinitesimal character determined by $\half (^{L}h)$ in the notation of  \cite[Section 5]{BVUni}, which is represented by $\lambda_{\check \CO}\in \fhh^*$. 

\begin{proposition}\label{counting}
   Suppose $\check \CO$ is a nilpotent $\check G$-orbit in $\check \g$. 
   Then 
   \[
    \abs{\Unipgen(\wtG)}= \sum_{\sigma \in \LCO}[\sigma : \Coh_{\Lambda_1}(\cK(G_0))],
   \]
   where 
  \begin{itemize}
        \item $\Lambda_1 = \lambda_\ckcO + \halfone + Q_0$, and  
      \item $\LCO \subset \Irr(W(\Lambda_1))$ is the Lusztig left cell attached to $\lambda_\ckcO$ (see \cite[Definition~3.34]{BMSZ1}).  
  \end{itemize} 
\end{proposition}

Here and henceforth, $\Irr (E)$ denotes the set of isomorphism classes of irreducible representations of a finite group $E$; $[\ : \ ]$ indicates the multiplicity of the first
(irreducible) representation in the second representation, of a finite group. If it is necessary to specify the finite group, a subscript will be included. 

\begin{proof}
   Let $\Lambda_2 := \lambda_\ckcO+Q_0$ and $\Lambda := \lambda_\ckcO + Q = \Lambda_1\sqcup \Lambda_2$. 
   We apply \cite[Corollary 5.4]{BMSZ1} twice:
   \[
   \begin{split}
    \abs{\Unipgen(\wtG)} &= \abs{\Unip_\ckcO(\wtG)} - \abs{\Unip_\ckcO(G_0)}\\
     & =   \sum_{\sigma \in \LCO}[\sigma : \Coh_{\Lambda_2}(\cKG)]
     - \sum_{\sigma \in \LCO}[\sigma : \Coh_{\Lambda_2}(\cK(G_0))]\\
     & =   \sum_{\sigma \in \LCO}\left([\sigma : \Coh_{\Lambda}(\cKG)] - [\sigma : \Coh_{\Lambda_2,\Lambda_1}(\cKG)]\right)\\
     & =  \sum_{\sigma \in \LCO}[\sigma : \Coh_{\Lambda_1,\Lambda_2}(\cKG)]   \quad \text{(by \eqref{eq:coheq})}\\
     & =  \sum_{\sigma \in \LCO}[\sigma : \Coh_{\Lambda_1}(G_0)]. \hspace {2.0in}\Box
     \end{split} 
   \]
\end{proof}

Before we proceed to count explicitly genuine special unipotent representations case by case, 
we recall the following notations and definitions in \cite{BMSZ1}. The relevant labels (which specify the types of groups we consider) are $\star =B,D,D^*$.

Denote by $\sfS_n\subset \GL_n(\Z)$ the group of the permutation matrices. The Weyl group
\[
    W= \begin{cases}
    \sfW_n, &  \text{if $\star = B$}; \\
      \sfW_n' , &  \text{if $\star \in \{D, D^*\}$}, \\
  \end{cases}
\]
where $\sfW_n\subset \GL_n(\Z)$ is the subgroup generated by $\sfS_n$ and all the diagonal matrices with diagonal entries $\pm 1$, and $\sfW'_n\subset \GL_n(\Z)$ is the subgroup generated by $\sfS_n$ and all the diagonal matrices with diagonal entries $\pm 1$ and determinant $1$. 

As always, $\sgn$ denotes the sign character (of an appropriate Weyl group).
By inflating the sign character of $\sfS_{n}$ (viewed as a quotient of $\sfW_n$), we obtain a quadratic character of $\sfW_{n}$, to be denoted by $\bsgn$. 
In addition let  
\[
\sfH_{t} := \sfW_t\ltimes \set{\pm 1}^t, \quad (t\in\bN),
\]
to be viewed as a subgroup in $\sfW_{2t}$, and let $\eta$ be the quadratic character of $H_t$, as in \cite[Section 8.1]{BMSZ1}. 

We define
\begin{equation}
\label{eq:Cohb}
  \cCb_{\star,n}:=
  \begin{cases}
    \bigoplus_{\substack{2t+c+d=n}} \Ind_{\sfH_{t} \times \sfW_{c}\times \sfW_{d}}^{\sfW_{n}}
    \eta\otimes 1\otimes 1, &  \text{if $\star =B$}; \medskip\\
    \bigoplus_{\substack{2t+a=n}} %\Res_{\sfW_{n}}^{\sfW'_{n}} \left(
    \Ind_{\sfH_{t} \times \sfS_a}^{\sfW_{n}}\eta\otimes 1, &  \text{if $\star = D $};\medskip \\
    % \Res_{\sfW_{n}}^{\sfW'_{n}} \left(
    \Ind_{\sfH_{\frac{n}{2}}}^{\sfW'_{n}}\eta, &  \text{if $\star = D^* $}.
  \end{cases}
\end{equation}
We also define
\begin{equation}\label{eq:Cohg}
  \cCg_{p,q} :=
  \begin{cases}
\displaystyle\bigoplus_{\substack{0\leq p-(2t+a+2r)\leq 1,\\0\leq q - (2t+a+2s)\leq 1}} \hspace{-2em} \Ind_{\sfH_{t} \times \sfS_{a}\times \sfW_s\times \sfW_r}^{\sfW_{n}}
    \eta \otimes 1 \otimes \sgn \otimes \sgn, &  \text{if $p+q=2n+1$}; \medskip\\
    \displaystyle \bigoplus_{\substack{2t+c+d+2r=p \\2t+c+d+2s=q}} \hspace{-1em}
    \Ind_{\sfH_{t} \times \sfW_s\times \sfW_r\times \sfW'_{c}\times \sfW_{d} }^{\sfW_{n}}\eta \otimes \bsgn \otimes \bsgn \otimes 1\otimes
    1, &  \text{if $p+q=2n$},\medskip \\
  \end{cases}
\end{equation}
and 
\begin{equation}\label{eq:CohgD*}
  \cCg_{D^{*},n} := \bigoplus_{\substack{2t+a=n}} \Ind_{\sfH_{t} \times \sfS_{a}}^{\sfW'_{n}} \eta \otimes \sgn, \qquad \qquad \text{if
    $\star =D^*$}.
\end{equation}

%Denote by $\Irr(\sfW_n)$ the set of isomorphism classes of irreducible representations of $\sfW_n$. 
We introduce some notations relevant to irreducible representations and 
Young diagrams. As usual, we identify $\Irr(\sfW_n)$ with the set of bipartions (or a pair of Young diagrams) of total size $n$ (\cite[Section 11.4]{Carter}). In what follows, we let $\ycol{a_1,a_2, \cdots, a_k}$ (resp. $\yrow{a_1,a_2, \cdots, a_k}$) denote the Young diagram whose $i$-th column (resp. row) has length $a_i$ if $1\leq i\leq k$ and has length $0$ otherwise.  
Let $\bfrr_i(\ckcO)$ denote the length of $i$-th row of the Young diagram of a nilpotent orbit $\ckcO$ of a classical group.

\subsection{Real odd spin groups}
\def\nb{n_{\mathbf b}}
\def\ng{n_{\mathbf g}}
\def\ycol#1{[#1]_{\text{col}}}
\def\yrow#1{[#1]_{\text{row}}}
\def\PP{\mathrm{PP}}
\def\tvarepsilon{\tilde{\varepsilon}}

In this case  $\wtG = \Spin(p,q)$ and $G_0 = \SO(p,q)$, with $p+q=2n+1$.
Let $N$ be the smallest positive integer such that $\bfrr_{2N}(\ckcO) = 0$.
Suppose the multiset $\set{\bfrr_i(\ckcO): 1\leq i \leq  2N}$ is  
\[
\set{2r'_{1}+1, 2r'_{1}+1,2r'_{2}+1,2r'_{2}+1,\cdots,2r'_{l}+1,2r'_{l}+1,
2r''_{1},2r''_{2}, \cdots,2r''_{2k}}.
\]
with $2l+2k = 2N$,  
$r'_i\geq r'_{i+1}$ (for $1\leq i < l$) and $r''_i\geq r''_{i+1}$ (for $1\leq i < 2k$). 
Note that $r''_{2k} = 0$ by definition. 

Let 
\[
\nb = l+\sum_{i=1}^l 2r'_i \qquad \text{ and } \qquad
\ng = \sum_{i=1}^{2k} r''_i. 
\]
Then \[
\Lambda_1  = (\underbrace{\half, \cdots \half}_{\nb\text{-terms}},\underbrace{0, \cdots 0}_{\ng\text{-terms}} ) + \bZ^n \qquad \text{ and } \qquad
W(\Lambda_1) = \sfW_{\nb}\times \sfW_{\ng}. 
\]

As a $W(\Lambda_1)$-module, the coherent continuation representation is   
\[
   \Coh_{\Lambda_1}(\cK(G_0)) =
   \begin{cases}
       \cCg_{p-\ng,q-\ng} \otimes \cCb_{B,\ng},& \text{if $p,q\geq \ng$;}\\
       0, & \text{otherwise.}
   \end{cases}
\]
See \cite[Propositions 8.1 and 8.2]{BMSZ1}. 

As in \cite[Section 2.8]{BMSZ1}, let
\[
\PP(\ckcO) :=   \set{(2i, 2i+1) | r''_{2i}>r''_{2i+1} \, \text{ and }\, 1\leq i< k}
\]
be the set of primitive pairs attached to $\ckcO$. The Lusztig left cell $\LCO$ in $\Irr(W(\Lambda_1))$ is given by 
\[
\LCO = \set{\tau_b \otimes \tau_\wp | \wp \subseteq \PP(\ckcO)},
\]
where 
\[
    \tau_b = \ycol{r'_{1}+1,r'_{2}+1, \cdots, r'_{l}+1} \times \ycol{r'_{1},r'_{2}, \cdots, r'_{l}}, 
\]
and  
\[
    %\tau_\wp = \ycol{r_{\wp,2},\cdots ,r_{\wp, 2k}}\times\ycol{r_{\wp,1}, r_{\wp,3}, \cdots, r_{\wp, 2k+1}} 
    \tau_\wp = \ycol{l_{\wp,1},l_{\wp,2},\cdots ,l_{\wp, k+1}}\times\ycol{r_{\wp,1}, r_{\wp,2}, \cdots, r_{\wp, k+1}}. 
\]
Here $(l_{\wp,1},l_{\wp,k+1}) = (r''_1,0)$ and for $1< i \leq k$,    
\[
    (l_{\wp,i}, r_{\wp,i+1}) = \begin{cases}
        (r''_{2i+1},r''_{2i}), & \text{$(2i,2i+1) \in \wp$};\\
        (r''_{2i},r''_{2i+1}), & \text{otherwise}. 
    \end{cases}
\]
%\[
%    (l_{\wp,i}, r_{\wp,i+1}) = \begin{cases}
%        (r''_{2i},r''_{2i+1}), & \text{$(2i,2i+1) \not\in \wp$};\\
%        (r''_{2i+1},r''_{2i}), & \text{otherwise}. 
%    \end{cases}
%\]
See \cite[Proposition 8.3]{BMSZ1}. 

%Then
%$\tau_{b} = [ r'_{1}+1,r'_{2}+1, \cdots, r'_{l}+1 ]_{col}\times [r'_{1},r'_{2}, \cdots, r'_{l}]_{col}$.

\begin{lemma}
We have 
  \[
  [\tau_{b}:\cCg_{p-\ng,q-\ng}] =  \begin{cases}
      1, & \text{ if $|p-q|=1$;}\\
      0, & \text{otherwise.}
  \end{cases}
  \]
\end{lemma}
\begin{proof}
    The multiplicity is computed by the Littlewood-Richardson rule which is reformulated in terms of counting painted bipartitions as in \cite[Section~2.8]{BMSZ1}. 
  % According to that  the only possible painting on $\tau_b$ is given by 
  % is putting ``$c$'' at the bottom of each column of the left diagram and put
 % dots anywhere else. In this case $|p-q|=1$. 
 % (i.e. $\Spin(p,q)$ is split).
\end{proof}

\begin{lemma}
For $\wp\subseteq \PP(\ckcO)$, we have
\[
 [\tau_{\wp}:\cCb_{B,n_{g}}] =  \begin{cases}
     1,  & \text{if $\wp = \emptyset$ and $\bfrr_{2i+1}(\ckcO) = \bfrr_{2i+2}(\ckcO)$
     for all $i\in \bN$;}\\
     0,  & \text{otherwise.}
 \end{cases}
\]
  %$[\tau_{\wp}:\cCb_{n_{g}}] =  1$ only if $r_{2i-1} = r_{2i}$.
\end{lemma}

\begin{proof}
  Note that we always have
  $l_{\wp,i} \leq r_{\wp,i}$ for $1\leq i \leq k+1$. 
%  $ \bfcc_{i}(\tau_{\wp,L}) \leq \bfcc_{i}(\tau_{\wp,R})$
\trivial[h]{
Note that 
$l_{\wp,i} = r_{2i+1} $ or $r_{2i}$,  $r_{\wp,i} = r_{2i-1}$ or $r_{2i}$,  and  $r_{2i+1}\leq r_{2i}\leq r_{2i-1}$.
}
  %since $\bfcc_i(\tau_{\wp,L}) = r_{2i+1} $ or $r_{2i}$,  $\bfcc_i(\tau_{\wp,R}) = r_{2i-1}$ or $r_{2i}$,  and  $r_{2i+1}\leq r_{2i}\leq r_{2i-1}$.
  Again the multiplicity is computed by the Littlewood-Richardson rule and the result can be formulated in terms of counting certain painted bipartitions. In a bit more detail,  from the painting rules of type $B$ in \cite[Definition~8.5]{BMSZ1}, the symbols ``$c$'' and ``$d$'' are painted in the left diagram of $\tau_\wp$. The only possible painting allowed on $\tau_\wp$ will be to put the symbol ``$\bullet$'' in every box, in which case we will have $r''_{2i-1}=r_{\wp,i} = l_{\wp,i} = r''_{2i}$, $\wp=\emptyset$ and the multiplicity is $1$.
\end{proof}

Together with Proposition \ref{counting}, the above two lemmas imply Theorem \Cref{thm:Spinpq0} for real odd spin groups.

\subsection{Real even spin groups and quaternionic spin groups}

We now assume $G_0 = \SO(p,q)$ with $p+q=2n$, or $G_0=\SO^*(2n)$. Let $N$ be the smallest positive integer such that $\bfrr_{2N+1}(\ckcO) = 0$.
Suppose the multiset $\set{\bfrr_i(\ckcO): 1\leq i \leq  2N}$ is  
\[
\set{2r'_{1}, 2r'_{1},2r'_{2},2r'_{2},\cdots,2r'_{l},2r'_{l},
2r''_{1}+1,2r''_{2}+1, \cdots,2r''_{2k}+1}.
\]
with $2l+2k = 2N$,  
$r'_i\geq r'_{i+1}$ (for $1\leq i < l$) and $r''_i\geq r''_{i+1}$ (for $1\leq i < 2k$). 
%Note that $r_{2k} = 0$ by definition. 

%Then $\ckG = \mathrm{PSO}(2n,\bC)$. The nilpotent orbits of $\ckG$ are still
%parameterized by partitions with labels I/II.

%Suppose $\ckcO'_{b} = [2r'_{1}, 2r'_{2},\cdots,2r'_{l}]_{row}$,
%and $\ckcO_{g} = [2r_{1}+1,2r_{2}+1, \cdots,2r_{k}+1]_{row}$.

Let 
\[
\nb = \sum_{i=1}^l 2r'_i \qquad \text{ and } \qquad 
\ng = k+\sum_{i=1}^{2k} r''_i. 
\]
Then 
\[
\Lambda_1  = (\underbrace{0, \cdots 0}_{\nb\text{-terms}},\underbrace{\half, \cdots \half}_{\ng\text{-terms}}) \qquad \text{ and } \qquad
W(\Lambda_1) = \sfW'_{\nb}\times \sfW'_{\ng}. 
\]

Recall that irreducible representations of $\sfW_n'$ are parameterized by unordered pairs of bipartitions with a label $I$ or $II$ attached when the pairs are equal to each other.
%To ease the notation, we also attach $I$ to any unordered pairs of two different partitions. 
For the labeling convention, see \cite[Section 8.3]{BMSZ1}.   

As in \cite[Section 2.8]{BMSZ1}, let
\[
\PP(\ckcO) :=   \set{(2i, 2i+1) |  r''_{2i}>r''_{2i+1}>0 \text{ and } 1\leq i< k}
\]
be the set of primitive pairs attached to $\ckcO$. The Lusztig left cell $\LCO$ in $\Irr(W(\Lambda_1))$ is given by  
\[
\LCO = \set{\tau_b \otimes \tau_\wp | \wp \subseteq \PP(\ckcO)}, 
\]
where  
\[
    \tau_b = \set{\ycol{r'_{1},r'_{2}, \cdots, r'_{l}}, \ycol{r'_{1},r'_{2}, \cdots, r'_{l}}}_{I/II}
\]
(the label $I/II$ of $\tau_b$ is determined by $\ckcO$)
and  
\[
    %\tau_\wp = \ycol{r_{\wp,2},\cdots ,r_{\wp, 2k}}\times\ycol{r_{\wp,1}, r_{\wp,3}, \cdots, r_{\wp, 2k+1}} 
    \tau_\wp = \set{\ycol{l_{\wp,1},l_{\wp,2},\cdots ,l_{\wp, k}}, \ycol{r_{\wp,1},r_{\wp,2},\cdots ,r_{\wp, k}}}.  %\tau_{\wp,L},\tau_{\wp,R}}. 
\]
Here $(l_{\wp,1},r_{\wp,k}) = (r''_1+1,r''_{2k})$ and for $1\leq  i < k$,  
\[
    (l_{\wp,i+1}, r_{\wp,i}) = \begin{cases}
        (r''_{2i}+1,r''_{2i+1}), & \text{$(2i,2i+1) \in \wp$};\\
         (r''_{2i+1}+1,r''_{2i}), & \text{otherwise}. 
    \end{cases}
\]
%\[
%    (l_{\wp,i+1}, r_{\wp,i}) = \begin{cases}
%        (r''_{2i+1}+1,r''_{2i}), & \text{$(2i,2i+1) \notin \wp$};\\
%        (r''_{2i}+1,r''_{2i+1}), & \text{otherwise}. 
%    \end{cases}
%\]
See \cite[Proposition 8.3]{BMSZ1}. 

The following lemma is clear. 
\begin{lemma}\label{lem:lr}
For $1 \leq i \leq k$ and $\wp\subseteq \PP(\ckcO)$, we have
    %\begin{equation}\label{eq:ref}
    \[
    l_{\wp,i} > r_{\wp,i}. 
    \]
    %\end{equation}
  \trivial[h]{
  Case  $(2i-2,2i-1)\notin \PP(\ckcO)$:
  \begin{itemize}
  \item  If $(2i,2i+1)\notin \wp$, then $l_{\wp,i} = r_{2i-1}+1 > r_{2i}=r_{\wp,i}$. 
  \item   If $(2i,2i+1)\in \wp$, then $l_{\wp,i} = r_{2i-1}+1 > r_{2i} \geq r_{2i+1}+1=r_{\wp,1}$. 
\end{itemize}

  Case  $(2i-2,2i-1)\in \PP(\ckcO)$:
  \begin{itemize}
  \item  If $(2i,2i+1)\notin \wp$, then $l_{\wp,i} = r_{2i-2}+1 > r_{2i}=r_{\wp,i}$. 
  \item  If $(2i,2i+1)\in \wp$, then $l_{\wp,i} = r_{2i-2}+1 > r_{2i} \geq r_{2i+1}+1=r_{\wp,i}$. 
\end{itemize}
  }
\end{lemma}

%Then $\tau_{b} =[ r'_{1},r'_{2}, \cdots, r'_{l}]_{col}\times [r'_{1},r'_{2}, \cdots, r'_{l}]_{col,I}$.

\subsubsection{The case when $G_0=\SO(p,q)$, $p+q=2n$}
%Let $G = \wtSpin(p,q)$ be the spin cover of $\SO(p,q)$.

In this case, the coherent continuation representation as a $W(\Lambda_1)$-module is
\[
   \Coh_{\Lambda_1}(\cK(G_0)) =
   \begin{cases}
       \cCg_{p-\ng,q-\ng} \otimes \cCb_{\ng},& \text{if $p,q\geq \ng$};\\
       0, & \text{otherwise.}
   \end{cases}
\]
Here the right-hand side is understood as its restriction to $W(\Lambda_1)=
\sfW'_{\nb}\times \sfW'_{\ng}$. See \cite[Propositions 8.1 and 8.2]{BMSZ1}. 

\begin{lemma}
    We have 
  \[ 
  [\tau_{b}:\cCg_{p-\ng,q-\ng}] = 
  \begin{cases}
      1, & \text{if $p=q$};\\
      0, & \text{otherwise.}
  \end{cases}
  \]
\end{lemma}
\begin{proof}
    Since the pair of Young diagrams in $\tau_b$ have the same shape, 
    $\tau_b$ does not occur in any of the terms of \eqref{eq:Cohg} with any of $s,r,c,d$ non-zero. 
    On the other hand, $\tau_b$ occurs with multiplicity one in 
    $\Ind_{\sfH_{\frac{\nb}{2}}}^{\sfW_{\nb}} \eta $ (see \cite[(8.15)]{BMSZ1}, for example).
\end{proof}

\begin{lemma}\label{lem:Sppq1}
Suppose $\ng > 0$ and $\wp\subseteq \PP(\ckcO)$. 
Then 
  \[ 
  [\tau_{\wp}:\cCb_{D,\ng}] = 
  \begin{cases}
      2, & \text{if $\wp=\emptyset$ and $r''_{2i+1} = r''_{2i+2}$ for all $i\in \bN$};\\
      0, & \text{otherwise.}
  \end{cases}
  \]
  %If $r_{2i-1} = r_{2i}$ for all $i$ and $\wp = \emptyset$,
  %then $[\tau_{\wp}:\cCb_{n_{g}}]_{\sfW'_{n_g}} =  2$ . 
  %Otherwise $[\tau_{\wp}:\cCb_{n_{g}}]_{\sfW'_{n_g}} =  0$. 
\end{lemma}
\begin{proof}
  By Lemma~\ref{lem:lr}, $\tau_{\wp}=\tau_{\wp,L}\times \tau_{\wp,R}$, with $\tau_{\wp,L}\ne \tau_{\wp,R}$. 
  %  is given by a pair of Young diagrams with different shapes.
 Therefore, \[\Ind_{\sfW'_{\ng}}^{\sfW_{\ng}} = \tau_{\wp,L}\times \tau_{\wp,R} \oplus  (\tau_{\wp,L}\times \tau_{\wp,R}\otimes \varepsilon),\]  
 where $\varepsilon$ is the quadratic character of $\sfW_{\ng}$ whose kernel is $\sfW'_{\ng}$.  
 Now 
  \[
  \begin{split}
    [\tau_\wp:\cCb_{D,\ng}]_{\sfW'_{\ng}}  & = 
    [\tau_{\wp,L}\times \tau_{\wp,R}: \cCb_{D,\ng}]_{\sfW_{\ng}}
    + [\tau_{\wp,L}\times \tau_{\wp,R}\otimes \varepsilon:\cCb_{D,\ng}]_{\sfW_{\ng}}\\
    & = 2  [\tau_{\wp,L}\times \tau_{\wp,R}: \cCb_{D,\ng}]_{\sfW_{\ng}}
  \end{split}
  \]
  since $\cCb_{D,\ng} \otimes \varepsilon  = \cCb_{D,\ng} $. 
   It is easy to see that (by the painting rules of type $D$ in \cite[Section 2.8]{BMSZ1}, for example), 
 $\tau_{\wp,L}\times \tau_{\wp,R}$ can only occur in the term of \eqref{eq:Cohb} with $a = k$. 
    When this is the case, it occurs with multiplicity one and 
    \begin{equation}\label{eq:lr}
        l_{\wp,i} = r_{\wp,i}+1 \quad \text{for $1\leq i \leq k$.}
    \end{equation}
    The equation \eqref{eq:lr} forces $r''_{2i-1} = r''_{2i}$ for $1\leq i \leq k$ and $\wp = \emptyset$.  
    This completes proof of the lemma. 
  % Without loss of generality, we may assume $\bfcc_1(\tau_{\wp,L}) > \bfcc_1(\tau_{\wp,R})$.  
  % We always have
  % $r_{2i+1}+1$ or
  % $r_{2i}+1 = \bfcc_{i}(\tau_{\wp,L}) > \bfcc_{i}(\tau_{\wp,R}) = r_{2i+2}-1$
  % or $r_{2i+3}-1$.
   \trivial[h]{
  By painting rules, the only possible painting on $\tau_{\wp,L}\times \tau_{\wp,R}$ is to put  ``c'' at the bottom
  of each column in the left diagram and  dots in other places.
  In this case, $r_{2i+1}=r_{2i+2}$, $\wp=\emptyset$. 
  }% and the multiplicity are $1$.
  %Since we are doing branching from $\sfW_{n_g}$-module to $\sfW'_{n_g}$, we see that  
  %$\tau_{\wp}$ occurs with multiplicity $2$. 
\end{proof}

Together with Proposition \ref{counting}, the above two lemmas imply Theorem~\ref{thm:Spinpq0} for real even spin groups.  

\subsection{The case when $G_0=\SO^*(2n)$}
\label{secpf1star}
%\subsubsection{Quaternionic spin groups}

%Suppose $G = \Spin^{*}(2n)$. We the notion of $n_{g}:= \abs{\ckcO_{g}}/2$ and
%$n_{\ckcO'_{b}}$.
In this case, the coherent continuation representation as a $W_{\Lambda_1}$-module is 
\[
  \Coh_{\Lambda_1}(G_0)  = 
  \Ind_{\sfW'_{\ng}\times \sfW'_{\nb}}^{W_{\Lambda_1}}
  \cCg_{D^*, \nb}\otimes \cCb_{D^*, \ng},
\]
where 
\[
W_{\Lambda_1} = (\sfW_{\nb}\times \sfW_{\ng}) \cap \sfW'_{n}. 
\]
See \cite[Propositions 8.1 and 8.2]{BMSZ1}. 

%We retain the notation $\tau_b$, $\tau_\ng$ 
\begin{lemma}
    If $\ng\neq 0$, then $[\tau_b\otimes \tau_{\wp}: \Coh_{\Lambda_1}(G_0)]= 0$ for each $\wp\in \PP(\ckcO)$. 
\end{lemma}
\begin{proof}
    Applying the Littlewood-Richardson rule (see \cite[(8.15)]{BMSZ1}) to $\cCb_{D^*,\ng}$ (see \eqref{eq:Cohb}), we see that if $\sigma_1\otimes \sigma_2\in \Irr(\sfW_{\nb}\times \sfW_{\ng})$
    occurs in $\Coh_{\Lambda_1}(G)$, then the pair of partitions in $\sigma_2$ must have the same shape. On the other hand, $\tau_{\wp}$ is represented by a pair of partitions with different shapes by Lemma \Cref{lem:lr}.
    This proves the lemma.  
\end{proof}

 We are now in the case when $\ng = 0$ (and so $n_b=n$). We have
 \begin{equation}\label{eq:Dstar}
  \LCO = \set{\tau_b} \qquad \text{and} \qquad \Coh_{\Lambda_1}(G_0) = \cCg_{D^*, n}.     
 \end{equation}

Recall the notion of $G$-relevant orbit in \eqref{def:relevant}.
 
\begin{lemma}\label{lem:Spin*b}
Suppose $\ng=0$. Then 
    \[
    [\tau_b: \cCg_{D^*, n}]=
    \begin{cases}
        \prod_{i=1}^{l}(r'_i-r'_{i+1}+1), & \text{if $\ckcO$ is $G$-relevant;}\\
        0, & \text{otherwise.} 
    \end{cases}
    \]
    In all cases, it is equal to  $\sharp ( G\backslash \sqrt{-1} \g^* \cap \CO)$. 
    (Here $r'_{l+1} := 0$ by convention. )
\end{lemma}
\begin{proof}
        By the Littlewood-Richardson rule (\cite[(8.15)]{BMSZ1}),
        all representations occurring in $\cCg_{D^*, n}$ (see \eqref{eq:CohgD*}) are labeled by $I$.
       %Note that the subgroup $\sfH_t\times \sfS_a$ in \eqref{eq:CohgD*} are in the Weyl group of the Levi subgroup of the complexification of $P$.  
        So if $\tau_b$ is to occur in it, $\tau_b$ must be labeled by $I$, which is equivalent to $\ckcO $ being $G$-relevant. 
      %   all representations occurring in $\cCg_{D^*, n}$ are labeled by $I$.
    %Suppose $\ckcO$ is relevant. 
    The rest of the claim about the multiplicity of $\tau_b$ in $\cCg_{D^*, n}$ also follows from the Littlewood-Richardson rule. 
    Concretely, the multiplicity is counted by the number of paintings on the bipartition of shape $\ycol{r'_1,\cdots, r'_l}\times \ycol{r'_1,\cdots, r'_l}$, using the painting rules of type $D^*$ in \cite[Section 2.8]{BMSZ1}. 
    It is routine to check that there are  
         $\prod_{i=1}^{l}(r'_i-r'_{i+1}+1)$  such paintings in total. 

    %such that at most one ``s'' can be painted on each row and dots are painted in the other places such that boxes painted by dots form two Young diagrams with the same shape. 

When $\ckcO$ is $G$-relevant, $\prod_{i=1}^l(r'_{i}-r'_{i+1}+1)$ also counts the number of signed Young diagrams having shape $\cO$, which is the number of real nilpotent orbits in $\cO$. When $\ckcO$ is not $G$-relevant, 
$\sqrt{-1} \g^* \cap \CO$ is empty. The lemma thus follows.        
         \end{proof}
  % It is clear that all paintings have the form that some ``s''
  %are painted at the end of each column in the left diagram, for $i$-th column at most
  %$(r'_{i}-r'_{i+1})$ times. In the right diagram, $r$ is painted if and only if $s$
  %occurs in the corresponding place in the left diagram.

    \trivial[h]{
    Translation to the painted bipartition: 
    \[
    \begin{split}
     [\tau_b,
     \bigoplus_{\substack{2t+a=n}} \Ind_{\sfH_{t} \times \sfS_{a}}^{\sfW'_{n}} \tilde{\varepsilon} \otimes \sgn]_{\sfW'_n}
     & =
     [\tau_b,
     \bigoplus_{\substack{2t+a=n}} \Ind_{\sfH_{t} \times \sfS_{a}}^{\sfW_{n}} \tilde{\varepsilon} \otimes \sgn]_{\sfW'_n}\\
     & = 
     [\Ind_{\sfW'_n}^{\sfW_n}\tau_b,
     \bigoplus_{\substack{2t+a=n}} \Ind_{\sfH_{t} \times \sfS_{a}}^{\sfW_{n}} \tilde{\varepsilon} \otimes \sgn]_{\sfW_n}\\
     & = 
     [\tau_{b,L}\times \tau_{b,R},
     \bigoplus_{\substack{2t+a=n}} \Ind_{\sfH_{t} \times \sfS_{a}}^{\sfW_{n}} \tilde{\varepsilon} \otimes \sgn]_{\sfW_n}\\
    \end{split} 
    \]
    The last term is counted by paintings. 
    }
%Let $\cO = [2r'_{1},2r'_{1}, 2r'_{2},2r'_{2},\cdots,2r'_{l},2r'_{l}]_{col}$.
%Note that $\cO$ has codimension $2$-boundary.

\delete{
\begin{lemma}
  Suppose $n_{g}=0$. Then $[\tau_{b}:\cCg_{n_{b}}]_{W'_{n_{b}}} = \#\Nil_{\cO}(G)$.
\end{lemma}
\begin{proof}
%   Note that the shape of the left and right diagrams is identical.
%   It is clear that all paintings have the form that some ``s''
%   are painted at the end of each column in the left diagram, for $i$-th column at most
%   $(r'_{i}-r'_{i+1})$ times. In the right diagram, $r$ is painted if and only if $s$
%   occurs in the corresponding place in the left diagram.
%   In all other places, we paint dots.
% 
%   Now the number of possible painted bipartitions is 
%   \[
%     \prod_{i}(r'_{i}-r'_{i+1}+1).
%   \]
   
  Suppose $\ckcO$ is relevant. Then   
    $\# \Unip_\ckcO(G) =  \prod_{i=1}^l(r'_{i}-r'_{i+1}+1)$ 
    by Lemma~\ref{lem:Spin*b}. 
    This number is  the same as the number of valid signed Young diagrams having shape $\cO$ corresponding to the real nilpotent orbits in $\Lie(G)$.
  
    Suppose $\ckcO$ is not relevant. 
    Then both of $\#\Unip_\ckcO(G)$ and $\#\Nil_\cO(G)$ are zero.  
    This finishes the proof.
\end{proof}
}

%Since the set of nilpotent orbit in $\cO\cap \sqrt{-1}\fgg^*$
%is also counted by $\prod_{i=1}^l (r'_i-r'_{i+1}+1)$ (clear from the signed Young diagram classification), 

Together with Proposition \ref{counting}, the above two lemmas imply Theorem~\ref{thm:Spinpq0} for $G=\Spin^*(2n)$.

\section{Genuine special unipotent representations of real spin groups}
%We postpone the proof of Theorem \ref{thm:Spinpq0} to the last section.
Based on Theorem \ref{thm:Spinpq0}, we will prove Theorems \ref{thm:Spinpq03} to \ref{unip0} in this section and the next section.
We retain the notation and assumptions of Section \ref{sec:intro}.  

For the time being, $G$ can be either $\Spin(p,q)$ or $\Spin^*(2n)$. Recall 
the $\widetilde G$-representation $I(\chi)=\Ind_{\widetilde P}^{\widetilde G} \chi $, where $\chi$ is a genuine character on $\widetilde L$ of finite order. 

The following lemma follows from a general result of Barbasch 
\cite[Corollary~5.0.10]{B00}. See also \cite[Section 3]{MT}. 

\begin{lemma}\label{lemwf}
The wavefront cycle of $I(\chi)$ equals 
\[
\sum_{\mathbf o\in  \widetilde G\backslash \sqrt{-1} \g^* \cap \CO } \mathbf o. 
\]
\end{lemma}
\begin{proof}
%This follows from \cite[Corollary~5.0.10]{B00}. 
%\trivial[]{
Let $\fnn$ be the nilpotent radical of $\Lie (\widetilde P)$.
By  \cite[Corollary~5.0.10]{B00}, the wavefront cycle of $I(\chi)$ is given by 
\[
WF(I(\chi)) = \sum_{\widetilde G \cdot X} \frac{\sharp (C(\widetilde G_X))}{\sharp(C(\widetilde P_X))} \widetilde G\cdot X,
\]
%\[
%WF(I(\chi)) = \sum_{G \cdot X} \frac{|\pi_0(G_X)|}{|\pi_0(P_X)|} G\cdot X.
%\]
where the summation runs over the set of orbits of the form $\widetilde G \cdot X$ in
$\igg \cap \cO$ with $X\in \sqrt{-1} \mathfrak n^*$.  
Here $\widetilde G_X$ (resp. $\widetilde P_X$) denotes the centralizer of $X$ in $\widetilde G$ (resp. $\widetilde P$), and the symbol $C$ indicates the component group. 

It is routine to check that every orbit $\mathbf o \in \widetilde G\backslash \igg \cap \cO$ has a representative in $\sqrt{-1}\fnn^*$. Thus 
\[
WF(I(\chi)) = \sum_{\mathbf o\in  \widetilde G\backslash \sqrt{-1} \g^* \cap \CO } m_{\mathbf o}\mathbf o,
\]
where $m_{\mathbf o}=\frac{|C(\widetilde G_X)|}{|C(\widetilde P_X)|}$, with $X$ an element of $ \sqrt{-1}\fnn^* \cap \mathbf o$.

By a result of Kostant and Barbasch-Vogan \cite[Lemma 3.7.3]{CM},
$\widetilde G_X$ has a semidirect product decomposition: 
\[\widetilde G_X=R\ltimes U_X,\]
where the reductive part $R$ is the centralizer in $\widetilde G$ of the $\s\l_2$-triple containing $X$, and $U_X$ is the unipotant radical of $\widetilde G_X$.  

In view of condition (b) of Lemma \ref{lemf} and by replacing $\widetilde P$ by a $\widetilde G$-conjugate if necessary, the group $R$ will be contained in the Levi component $\widetilde L$ of $\widetilde P$ and therefore in $\widetilde P$ in particular. Since $\widetilde P_X=\widetilde P\cap \widetilde G_X$, we conclude that $\widetilde P_X = R\ltimes (\widetilde P\cap U_X)$. 
%(Note that $G_X = R \ltimes U_X$ where $U_X$ is the unipotent radical of $G_X$, Take $p \in C_P(X)$, then $p = r u$ with $r\in R$, $u\in U_X$. Since $r\in R\subset P$, we conclude that $u\in P$. In other word, 
The reductive part of $\widetilde P_X$ and $\widetilde G_X$ are therefore isomorphic and consequently,
\[
\frac{\sharp(C(\widetilde G_X))}{\sharp (C(\widetilde P_X))} = 1. 
\]
The lemma follows. \end{proof}

%It is easy to see that, $\bfoo\cap \fnn\neq \empty$ is for every real nilpotent orbit $\bfoo$ in $\cO\cap \igg$.  
\trivial[h]{
We can explicitly construct an element in $\bfoo\cap \fnn$. 
Recall that real nilpotent orbits are classified by signed Young diagrams.

First, suppose that $G_{\bC}$ is an even spin group.  
Then $\cO$ is very even, i.e., there is a non-negative integer $k$  such that 
$\bfrr_{2i+1}(\cO) = \bfrr_{2i+2}(\cO)>0$ is even for every $i<k$ and $\bfrr_{2k+1}(\cO)=0$.
For each real nilpotent orbit $\bfoo$ with the same Young diagram as $\cO$ and $X\in \bfoo$, one can construct a basis of $V$: 
\begin{equation}\label{eq:basis}
\Set{e_{i,l} |l \in \set{-\bfrr_{2i+1}(\cO)+1,-\bfrr_{2i+1}(\cO)+3 ,\cdots, \bfrr_{2i+1}(\cO)-1 }}_ {i<k}  
\end{equation}
such that
\begin{itemize}
    \item  $\inn{e_{i,l}}{e_{j,l'}} = \delta_{i,j} \delta_{l,-l'}$ for $i\in \bN$, $l,l'\in \bZ$,
    \item $X e_{i,l} = e_{i,l+2}$ (for $l\in \set{-\bfrr_{2i+1}(\cO)+1,-\bfrr_{2i+1}(\cO)+3 ,\cdots, \bfrr_{2i+1}(\cO)-3}$, 
    \item the form $\inn{\_}{\_}_j$ defined by $\inn{\_}{X^{j}\_}$ on 
    $V_j:=\sspan\set{e_{i,-2j+1}| \bfrr_{2i+1}(\ckcO) = 2j}$ 
    is non-degenerate ($V_j$ can be zero), 
    \item the reductive part of $G_X$ is isomorphic to the product $R$ of the isometry group of $\inn{\_}{\_}_j$ ($j$ runs over the set of positive integers). 
\end{itemize}

Suppose that $G_{\bC}$ is an odd spin group. Then there is a non-negative integer $k$ such that  then nonzero rows of $\cO$ satisfies   
$\bfrr_{2i+1}(\cO) = \bfrr_{2i+2}(\cO)>0$ is even for $i<k$, $\bfrr_{2k+1}(\cO) = 1$, and $\bfrr_{2k+2} = 0$. An isotropic vector $e_{k,0}$ together with \eqref{eq:basis} form a basis of $V$, with similar properties. 

Suppose $G$ is a quaternionic spin group.  By $G$-conjugations, we can assume that $P$ is the parabolic group stabilizes the flag $\set{F_{\geq j}}$
with 
\[
F_{\geq j} := \sspan\set{ e_{i,l} | l \geq j}. 
\]
We can take the Levi $L$ of $P$ to stabilizing the set of spaces $\set{F_j}$ with 
\[
F_j := \sspan\set{ e_{i,l} | l = j}.
\]
Suppose $G$ is a real special orthogonal group. Then there are two orbits of Lagrangian subspaces. 
And $P$ stabilizes the flag if and only if $G_\bC \cdot X$ is the induced orbit $\Ind_{P_\bC}^{G_\bC} 0$. Equivariantly, the label of $\cO$ is compatible with the Lagrangian $F_{\geq 1}$.

In the above setting, $R$ preserves the spaces $\set{F_j}$. 

%addition to above, we add an anisotropic vector $e_{k,0}$.
}

\begin{lemma}\label{lemwmf002}
The representation $I(\chi)$ is completely reducible and multiplicity free. Moreover, every  irreducible subrepresentation of $I(\chi)$ belongs to $\Unipgen(\widetilde G)$. 
    \end{lemma}
\begin{proof}
With the easy calculation of the infinitesimal character of $I(\chi)$, Lemma \ref{lemwf} implies that every irreducible subquotient of $I(\chi)$ belongs to $\Unipgen(\widetilde G)$. Since $I(\chi)$ is unitarizable, it is completely reducible. Lemma \ref{lemwf} implies that it is also multiplicity free. 
\end{proof}

In the rest of this section, we focus on the real spin groups. 

\begin{lemma}\label{lemwmf2}
Suppose that $G=\Spin(p,q)$ and $\sharp(\Unipgen(\widetilde G))=1$. Then $I(\chi)$ is irreducible and $\Unipgen(\widetilde G)=\{I(\chi)\}$. 
\end{lemma}
\begin{proof}
This is a direct consequence of Lemma \ref{lemwmf002}.    
\end{proof}

\begin{lemma}\label{lemwmf33}
Suppose that $G=\Spin(p,q)$ and $\sharp(\Unipgen(\widetilde G))=2$. Then $I(\chi)$ is irreducible and $\Unipgen(\widetilde G)=\{I(\chi), I(\chi')\}$. Here  $\chi'$ is a genuine character on $\widetilde L$ of finite order that is not conjugate to $\chi$ by $\mathrm N_{\widetilde G}(\widetilde L)$. 
\end{lemma}
\begin{proof}
The conditions of the lemma imply that $p=q$. As in the proof of Lemma \ref{lemf2}, let $Z$ denote the inverse image of $\{\pm 1\}$ under the covering homomorphism $\widetilde G\rightarrow \SO(p,q)$, which is a central subgroup of $\widetilde G$ and it is contained in $\widetilde L$. 
Then $\chi|_Z\neq \chi'|_Z$. Note that $Z$ acts on $I(\chi)$ and $I(\chi')$ through the characters $\chi|_Z$ and $\chi'|_Z$, respectively. Thus  $I(\chi)$ and $I(\chi')$ have no irreducible subrepresentation in common. In view of Lemma \ref{lemwmf002},  the lemma easily follows by the condition that  $\sharp(\Unipgen(\widetilde G))=2$.  
\end{proof}

In view of the counting result in Theorem \ref{thm:Spinpq0},  it is clear that Lemmas \ref{lemwmf2} and \ref{lemwmf33} imply part (a) of Theorem \ref{unip0}. Theorem \ref{thm:Spinpq03} is an easy consequence of part (a) of Theorem \ref{unip0}. By Clifford theory, Theorem \ref{thm:Spinpq4} is a direct consequence of Theorems \ref{thm:Spinpq0} and \ref{thm:Spinpq03}. 

\section{Genuine special unipotent representations of quaternionic spin groups}

In this section, we assume that $G=\Spin^*(2n)$. We will prove part (b) of Theorem \ref{unip0}. 
Recall that $\widetilde G=G$ and the set $\Unipgen(\widetilde G)$ is assumed to be nonempty. This implies that $n$ is even and $\check \CO$ is very even. Note that the Young diagram of $\CO$ is the transpose of that of $\check \CO$. Thus $\CO$ is also very even.  

Fix a Cartan involution $\theta$ on $G$ such that its fixed point group, to be denoted by $K$, is identified with $\widetilde{\oU}(n)$. Here $\widetilde{\oU}(n)$ is the double cover of $\oU(n)$ given by the square root of the determinant character on $\oU(n)$. The complexification of $K$, to be denoted by $K_\bC$, is identified with $\widetilde{\mathrm{GL}}_n(\C)$, the double cover of $\mathrm{GL}_n(\C)$ given by the square root of the determinant character on $\mathrm{GL}_n(\C)$.
The category of Casselman-Wallach representations of $G$ is equivalent to the category of $(\g_\bC,K_\bC)$-modules of finite length (see \cite[Chapter 11]{Wa2}). 
By using the trace form, we identify $\g_\bC^*$ with $\g_\bC$. Denote the Lie algebra of $K_\bC$ by $\k_\bC$, and write $\s_\C$ for its  orthogonal complement in $\g_\bC$ under the trace form.  

Let $\mathbf o'$ be a $K_\bC$-orbit in $\s_\C \cap \CO$. 
In what follows, we will first construct some auxiliary representations, to be precise irreducible $(\g_\C,K_\bC)$-modules $V_{\bfoo',N}$ ($N\in \bN$) with the associated variety $\baroop$.  
%The special unipotent representations we need are then obtained by the method of coherent continuation.  

Let $\set{\ree,\rhh,\rff}$ be an $\fsl_2$-triple in $\g_\bC$ such that $\rhh\in \fkk_\bC$, $\ree, \rff\in \s_\bC$, and $\rff\in \bfoo'$. 
%{\color{red} 
%Note that $\rhh$ is independent of the choice of $\fsl_2$-triple. 
%}
Then we have a decomposition
\[
\g_\bC=\bigoplus_{i\in \Z} \g_{\bC,2i}=\bar \u \oplus \l'_\bC\oplus \u,
\]
where $\g_{\bC,2i}$ is the eigenspace of the operator 
\[
\g_\bC\rightarrow \g_\bC, \quad x\mapsto [\mathrm h, x]
\]
with eigenvalue $2i$, and
\[
\l'_\bC:=\g_{\bC,0},\qquad \u:=\bigoplus_{i=1}^\infty \g_{\bC,2i}, \qquad \textrm{and}\qquad \bar\u:=\bigoplus_{i=1}^\infty \g_{\bC,-2i} .
\]
Write $\fqq=\l'_\bC\oplus \u$.
Then $\fqq$ is a $\theta$-stable parabolic subalgebra of $\g_\bC$. Note that
\[
 \l'_\bC\cong \g\l_{r_1}(\bC)\times \g\l_{r_2}(\bC)\times \dots \times \g\l_{r_k}(\bC).
\]
Write $C_\bC$ for the connected closed subgroup of $K_\C$ whose Lie algebra equals $\l'_\bC\cap \k_\bC$.

An easy calculation shows the following lemma. 

\begin{lemma}\label{rho0}
    Up to isomorphism, there is a unique one-dimensional $(\l'_\bC, C_\bC)$-module whose tensor square is isomorphic to $\wedge^{\dim \u} \u$. Moreover, this module is genuine. 
\end{lemma}

Let $\rho_\u$ be a  one-dimensional $(\l'_\bC, C_\bC)$-module as in Lemma \ref{rho0}. View it as a $(\mathfrak q,C_\bC)$-module via the trivial action of $\u$. Let $\mathcal R_\mathfrak q$ denote the functor from the category of $(\mathfrak q,C_\bC)$-modules to the category of $(\g,K_\bC)$-modules given by 
\[
  V_0\mapsto \mathrm{Hom}_{\mathcal H(\mathfrak q,C_\bC)}(\mathcal H(\g,K_\bC), V_0)_{K_\bC-\textrm{fin}}.
\]
Here $\mathcal H$ indicates the Hecke algebra, $\mathrm{Hom}_{\mathcal H(\mathfrak q,C_\bC)}(\mathcal H(\g,K_\bC), V_0)$ is viewed as an (left) $\mathcal H(\g,K_\bC)$-module via its right multiplication on itself, and  the subscript ${K_\bC-\textrm{fin}}$ indicates the $K_\bC$-finite parts. See \cite[Page 105]{KV} for more details. The functor $\mathcal R_\mathfrak q$ is left exact, and write $\mathcal R_\mathfrak q^i$ ($i\in \Z$) for its $i$-th derived functor.

%We fix a positive system of $\fgg_\bC$ compatible with $\fqq$. 
For each  integer $N$,  put
\[
V_{\bfoo',N} := \mathcal R_\mathfrak q^S( (\rho_\u)^{\otimes (2N+1)}),\qquad \textrm{where } S:=\dim (\k_\bC\cap \u).
\]
Write $\pi_{\bfoo',N}$ for the Casselman-Wallach representation of $G$ corresponding to $V_{\bfoo',N}$.

Fix a Cartan subalgebra $\ftt_\C\subset \fkk_\C$, which is also a Cartan subalgebra of $\fgg_\bC$. Fix an arbitrary positive system of the root system $\Delta(\ftt_\bC, \l_\bC')$, and write $\rho_{\l_\bC'}\in \ftt_\bC^*$ for the half sum of the positive roots. By abuse of notation, still write $\rho_\u\in \ftt^*_\C$ for the weight of the module $\rho_\u$.  

%in positive root system  Fix a positive   
%and a positive system of the root system $\Delta^+(\fgg_\bC, \ftt_\bC)$. 
%Let $\rho\in \ftt_\bC^*$ be the half sum of the positive roots. %$ = \sum_{\alpha\in \Delta^+(\fgg_\bC, \ftt_\bC)}$ 
\begin{lemma} \label{lemAq}
    When $N\geq 1$, $\pi_{\bfoo',N}$ is an irreducible genuine representation of $G$ with infinitesimal character $\rho_{\l_\bC'}+ 2N\rho_\u$ and associated variety $\baroop$. 
\end{lemma}
\begin{proof}
  Write $\rho:=\rho_{\l'_\bC}+\rho_\u$. By \cite[Corollary 5.100]{Knapp}, for every weight  $\alpha\in \t_\bC^*$ of $\fuu$, 
  \[
\inn{\rho}{\check\alpha}>0\qquad \textrm{and}\qquad \inn{\rho_\fuu}{\check\alpha}>0.
  \]
  Here $\check \alpha\in \t_\bC$ is the coroot corresponding to $\alpha$. Then
  \[
  \la \rho_{\l_\bC'}+ 2N\rho_\u, \check \alpha\ra>0,
  \]
  and hence 
    the $(\fll',C_\bC)$-module $(\rho_\fuu)^{\otimes (2N-1)}$ is in the good range (see \cite[Definition~0.4.9]{KV} for the definition of good range).   
    So $\pi_{\bfoo',N}$ is irreducible since cohomological induction in the good range preserves irreducibility (see \cite[Theorem~8.2]{KV}).  
    
   By a result of Barbasch-Vogan \cite[Proposition~3.4]{BV.W}, the associated variety of $\pi_{\bfoo',N}$ equals  $\overline{K_\C \cdot (\fss_\C \cap \bar\fuu )}$  (the Zariski closure).  See also \cite[Lemma~2.7]{Ko} and \cite[Proposition 5.4]{Tr}. 
    Note that $\overline{\cO} \subset \overline{ G_\bC \cdot \bar\fuu}$ since $f\in \bar\fuu$. Then by \cite[Theorem 7.3.3]{CM}, we conclude that    \[
    \overline{\cO} =\overline{ G_\bC \cdot \bar\fuu}.
    \]
    Similarly, $\baroop \subset \overline{ K_\bC \cdot (\fss_\C \cap \bar\fuu)}$ since $f\in \fss_\C \cap \bar\fuu$. Then by \cite[Corollary~5.20]{Vo89}, we conclude that 
    \[
    \overline{K_\C \cdot (\fss_\C \cap \bar\fuu )} = \baroop.
    \]
    %(see \cite[Corollary~5.20]{Vo89}). 
   % \trivial[]{
   %  We need the fact that the irreducible components of $\overline{\cO}\cap \fss_\C$ are $\overline{\bfoo'_i}$ where $\bfoo'_i$ running over $K_C$ orbits in $\cO\cap \fss_\C$.  
   % }
   % 
    The computation of the infinitesimal character of $\pi_{\bfoo',N}$ is straightforward (see \cite[Proposition~0.48]{KV}).  
    This finishes the proof.  
\end{proof}

\begin{lemma}
  %  Let $\bfoo'\in K_\bC \backslash(\fss_\bC\cap \cO)$.  
   There exists a representation in $\Unipgen(G)$ whose associated variety equals $\baroop$. 
\end{lemma}
\begin{proof}
   We follow an idea of Barbasch-Vogan in \cite{BV.W}. 
   Retain the notation in %\cite{BMSZ1} and 
   Section~\ref{secpf1}. 
   Recall that $\ckcO$ is very even.
   Let $\Lambda_2 = \lambda_\ckcO + Q_0\subset \fhh^*$. 
   It is easy to see that 
    $W(\Lambda_2) = \sfW'_n$ equals the abstract Weyl group $W$ of $G_\bC$ and  
   $\Set{\tau_b}$ is the double cell of $\Irr(W)$ associated to the  orbit $\cO$ via the Springer correspondence. See \eqref{eq:Dstar}. 
   
   For each $K_\bC$-stable closed subset $S$ in $\fss_\bC$,
    consider the $\sfW'_n$-submodule
    \[
   \Coh_{\Lambda_2,S}(\cKgen (G)):= \Set{ \Psi\in \Coh_{\Lambda_2}(\cKgen (G))|
    \AV(\Psi(\lambda))\subset S \text{ for all $\lambda\in \Lambda_2$}} 
    \]
    of the coherent continuation representation $\Coh_{\Lambda_2}(\cKgen (G))$. Here $\AV(\Psi(\lambda))$ denotes the union of associated varieties of irreducible genuine representations of $G$ occurring in $\Psi(\lambda)$ with non-zero coefficient. 
   %  Note that the module $A_\fqq(\rho_\fuu)$ is an irreducible $\wtG$-module, since $\rho_\fuu$ is in the good range. 
   % The infinitesimal character of $A_\fqq(\rho_\fuu)$ is $\rho_\fuu +\rho = \rho_\fll+ 2\rho_\fuu \in   \Lambda_2$. 
   % Moreover $\AV(A_\fqq(\rho_\fuu)) = \baroop$ by \cite[Proposition 5.4]{Tr}.   
  % $\rho_{\fll'}+2\rho_{\fuu}\in \Lambda_2$,
  
Let $\Delta^+(\t_\bC,\g_\bC)\subset \t_\bC^*$ denote the union of  $\Delta^+(\t_\bC,\l'_\bC)$ and the weights of $\fuu$, which is a positive system of the root system $\Delta(\t_\bC,\g_\bC)$. Using this positive system, we identify $\t_\bC$ with the abstract Cartan subalgebra $\h$ of $\g_\bC$. Note that $\rho_{\fll_\bC'} +2 \rho_{\fuu}$ belongs to $\Lambda_2$ and  is  regular in $\ftt_\C^*$.
%, and $\rho_{\fll_\bC'}+2\rho_{\fuu}$ is a regular element in $\ftt_\C$.
   
  Let  $\Psi_1$ be the element in $\Coh_{\Lambda_2}(\cKgen (G))$ such that $\Psi_1(\rho_{\fll_\bC'}+2\rho_\fuu) = \pi_{\bfoo',1}$. 
  Then $\Psi_1 \in \Coh_{\Lambda_2,\baroop}(\cKgen (G)) $
  %\setminus \Coh_{\Lambda_2,\partial \bfoo'}(\cKgen (G))$
  by Lemma~\ref{lemAq}.  

    Let $\bC[\fhh^*]$ denote the space of polynomial functions on $\fhh^*$. 
    By a result of Barbasch-Vogan and King \cite[Theorem~1.2]{King},  there is a well-defined $W$-equivariant linear map 
    \[
    \Coh_{\Lambda_2,\baroop}(\cKgen (G)) \rightarrow \bC[\fhh^*]
    \]
    sending $\Psi_1$ to the Goldie rank polynomial of the annihilator ideal of $V_{\baroop,1}$.   
    By a result of Joseph \cite[2.10]{Jo85}, the Goldie rank polynomial  generates an irreducible $W$-module which is isomorphic to $\tau_b$. 
 % Let $\sigma_\cO$ be the special representation of $W$ attached to $\cO$ via the Springer correspondence. 
 Therefore $\tau_b$ occurs in the $W$-module $\Coh_{\Lambda_2,\baroop}(\cKgen (G))$.
%  $ / \Coh_{\Lambda_2,\partial \bfoo'}(\cKgenG)$ is non-zero and
%  contains $\tau_b$ by a result due to Barabsch-Vogan and King see \cite[Lemma~4.7]{BMSZ1}.  
   % Applying \cite[Proposition~5.1]{BMSZ1}, we conclude that 

 Write $W_{\lambda_\ckcO}$ for the stabilizer of $\lambda_\ckcO$ in $W$, 
  and $1_{W_{\lambda_\ckcO}}$ for the trivial representation of $W_{\lambda_\ckcO}$. By \cite[Corollary 5.30]{BVUni},     $1_{W_{\lambda_\ckcO}}$ occurs in $\tau_b=(j_{W_{\lambda_\ckcO}}^W\sgn)\otimes \sgn$ with multiplicity-one.  Here $j$ stands for the $j$-operation, as in \cite[Section 11.2]{Carter}.    
  We therefore conclude that 
    \[
    \begin{split}
    & \sharp \Set{\pi\in \Unipgen(G)| \AV(\pi)=\baroop} \\
    = &\, \sharp  \Set{\pi\in \Unipgen(G)| \AV(\pi)\subset \baroop} \qquad \text{(by \cite[Theorem~8.4]{Vo89})}\\
    = & \sum_{\sigma \in \Irr(W) } [1_{W_{\lambda_\ckcO}}:\sigma] [\sigma: \Coh_{\Lambda_2,\baroop}(\cKgen (G)) ]
    \quad  \quad \text{(by \cite[Proposition~5.1]{BMSZ1})}\\
    \geq &  [\tau_b: \Coh_{\Lambda_2,\baroop}(\cKgen (G)) ]\geq 1.   
        % $\tau_b = j_{W_{\lambda_\ckcO}}^W(\sgn)\otimes \sgn$) 
    \end{split}
    \]
   Here $\AV$ stands for the associated variety. 
   This completes the proof. 
\end{proof}

Together with the counting result in Theorem \ref{thm:Spinpq0}, the above proposition implies that there is a unique representation $\pi_{\mathbf o'}\in \Unipgen(G)$ whose associated variety equals $\baroop$, and 
\[
  \Unipgen(G)=\{\pi_{\mathbf o'}\mid \mathbf o'\in K_\bC\backslash(\s_\bC\cap \CO)\}. 
\]
Recall that the wavefront cycle agrees with the weak associated cycle (\cite{SV}).  By  using Lemma \ref{lemwf} and  considering the wavefront cycle, we conclude that
\[
 I(\chi)\cong \bigoplus_{\mathbf o'\in K_\bC\backslash(\s_\bC\cap \CO)} \pi_{\mathbf o'}.
\]
This easily implies part (b) of Theorem \ref{unip0}. 

\begin{remark} In \cite{MT}, Matumoto and Trapa consider degenerate principal series representations of the linear group $\Sp(p, q), \SO^*(2n),$ or $\RU(m, n)$ with integral infinitesimal character. In particular, they prove that each irreducible constituent of maximal Gelfand–Kirillov dimension is a derived functor module. 
\end{remark}

%\begin{conjecture}
%Every element in $\Unipgen(\wtG)$ can be constructed via a cohomological induction. 
%\end{conjecture}

\begin{acknowledgement}
D. Barbasch is supported by NSF grant, Award Number 2000254. J.-J. Ma is supported by the National Natural Science Foundation of China (Grant No. 11701364 and Grant No. 11971305) and Xiamen University Malaysia Research Fund (Grant No. XMUMRF/2022-C9/IMAT/0019). B. Sun is supported by  National Key R \& D Program of China (No. 2022YFA1005300 and 2020YFA0712600) and the New Cornerstone Science Foundation.  C.-B. Zhu is supported by MOE AcRF Tier 1 grant R-146-000-314-114, and
Provost’s Chair grant E-146-000-052-001 in NUS. 

C.-B. Zhu is grateful to Max Planck Institute for Mathematics in Bonn, for its warm hospitality and conducive work environment, where he spent the academic year 2022/2023 as a visiting scientist. 

Some statements in the paper are (additionally) verified for low rank groups with the atlas software. J.-J. Ma thanks J. Adams for patiently answering questions on the atlas software. 
\end{acknowledgement}


\begin{thebibliography}{99.}%
% Use the following syntax and markup for your references if 
% the subject of your book is from the field 
% "Mathematics, Physics, Statistics, Computer Science"


\bibitem{ABV} Adams, J., Barbasch, D. and Vogan, D. A.: The Langlands classification and irreducible characters for real reductive groups, Progr. Math., vol. 104, Birkhauser, 1991. 

%\bibitem{ArPro} Arthur, J.: On some problems suggested by the trace formula, Lie group %representations, II (College Park, Md.), Lecture Notes in Math., vol. 1041, 1--49, 1984. 

\bibitem{AAM} Adams, J., Arancibia Robert, N. and Mezo, P.: Equivalent definitions of Arthur packets for real classical groups, arXiv:2108.05788. 

\bibitem{AM} Arancibia Robert, N. and Mezo, P.: Equivalent definitions of Arthur packets for real unitary groups, arXiv:2204.19715. 


\bibitem{ArUni} Arthur, J.: Unipotent automorphic representations: conjectures, Orbites unipotentes et repr\'esentations, II, Ast\'erisque (\textbf{171-172}), 13--71, 1989. 

\bibitem{ArEnd} Arthur, J.: The Endoscopic Classification of Representations: Orthogonal and Symplectic Groups, Amer. Math. Soc. Colloq. Publ., vol. 61, Amer. Math. Soc., Providence, RI, 2013. 

\bibitem{B89} Barbasch, D.: The unitary dual for complex classical Lie groups, Invent. Math. (\textbf{96}), no. 1, 103--176, 1989. 

\bibitem{B00} Barbasch, D.: Orbital integrals of nilpotent orbits,
    The mathematical legacy of {H}arish-{C}handra, Proc. Sympos. Pure Math.
    no. 68, 97--110, Amer. Math. Soc., Providence, RI,
    2000.

\bibitem{B17} Barbasch, D.: Unipotent representations and the dual pair correspondence, J. Cogdell et al. (eds.), Representation Theory, Number Theory, and Invariant Theory, In Honor of Roger Howe. Progr. Math., vol. 323, 47--85, 2017. 


\bibitem{BMSZ0} Barbasch, D., Ma, J.-J., Sun, B. and Zhu, C.-B.: On the notion of metaplectic Barbasch-Vogan duality, arXiv:2010.16089. To appear in Int. Math. Res. Not. IMRN. 

\bibitem{BMSZ1} Barbasch, D., Ma, J.-J., Sun, B. and Zhu, C.-B.: Special unipotent representations of real classical groups: counting and reduction, arXiv:2205.05266. 

\bibitem{BMSZ2} Barbasch, D., Ma, J.-J., Sun, B. and Zhu, C.-B.: Special unipotent representations of real classical groups: construction and unitarity, arXiv:1712.05552. 

\bibitem{BV.W} Barbasch, D. and Vogan, D. A.: Weyl Group Representations and Nilpotent Orbits, Representation Theory of Reductive Groups: Proceedings of the University of Utah Conference (1982), 21--33, Birkh{\"a}user Boston, 1983. 

\bibitem{BVUni} Barbasch, D. and Vogan, D. A.: Unipotent representations of complex semisimple groups, Annals of Math. (\textbf{121}), no. 1, 41--110, 1985. 

\bibitem{Brega} Brega, A.: On the unitary dual of $\Spin(2n,\BC)$, Trans. AMS, 
vol. 351, no. 1, 403-415, 1999. 

\bibitem{Carter} Carter, R. W.: Finite groups of Lie type, Wiley Classics Library, John Wiley \& Sons, Ltd., Chichester, 1993. 

\bibitem{CM} Collingwood, D. H. and McGovern, W. M.: Nilpotent orbits in semisimple Lie algebra: an introduction, Van Nostrand Reinhold Co., 1993. 

\bibitem{Howe79} Howe, R.: $\theta$-series and invariant theory, Automorphic Forms, Representations and $L$-functions, Proc. Sympos. Pure Math., vol. 33, 275--285, 1979. 

\bibitem{Howe89} Howe, R.: Transcending classical invariant theory, J. Amer. Math. Soc. (\textbf{2}), 535--552, 1989. 

\bibitem{Jo85} Joseph, A.: On the associated variety of a primitive ideal, J. Algebra 93 (\textbf{1985}), no. 2, 509–523 

\bibitem{King} King, D.: The character polynomial of the annihilator of an irreducible Harish-Chandra module, Amer. J. Math. (\textbf{103}),  1195–-1240, 1981.

\bibitem{Knapp} Knapp, A. W.: Lie groups beyond an introduction. Second edition. Progr. Math., vol. 140, Birkh\"auser Boston, Inc., Boston, Mass., 2002. xviii+812 pp.  

\bibitem{KV} Knapp, A. W. and Vogan, D. A.: Cohomological induction and unitary representations, Princeton Mathematical Series, vol. 45, Princeton University Press, Princeton, NJ, 1995.

\bibitem{Ko} Kobayashi, T: Discrete decomposability of the restriction of $A_{\mathfrak q}(\lambda)$ with respect to reductive subgroups. III. Restriction of Harish-Chandra modules and associated varieties. Invent. Math. (\textbf{131}), no. 2, 229–-256, 1998. 

\bibitem{LMBM} Losev, I., Mason-Brown, L. and Matvieievskyi, D.: Unipotent ideals and Harish-Chandra bimodules, arXiv:2108.03453. 

\bibitem{Lu} Lusztig, G.: Characters of reductive groups over a finite field, Ann. of Math. Stud., vol. 107, Princeton University Press, 1984. 

\bibitem{MT} Matumoto, H. and Trapa, P. E.: Derived functor modules arising as large irreducible constituents of degenerate principal series, Compos. Math. (\textbf{143}), no. 1, 222–-256, 2007. 

\bibitem{Sch} Schmid, W.: Two character identities for semisimple Lie groups, Non-commutative harmonic analysis (Actes Colloq., Marseille-Luminy, 1976), Lecture Notes in Math., Vol. 587, 196--225, Springer, Berlin, 1977. 

\bibitem{SV} Schmid, W. and Vilonen, K.: Characteristic cycles and wave front cycles of representations of reductive Lie groups, Ann. of Math. (\textbf{151}), no. 3, 1071--1118, 2000. 

\bibitem{SpVo} Speh, B. and Vogan, D. A.: Reducibility of generalized principal series representations, Acta Math. (\textbf{145}), no. 3--4,
227--299, 1980. 

\bibitem{Tr} Trapa, P. E.: Annihilators and associated varieties of $A_{\mathfrak{q}}(\lambda)$ modules for $\oU(p,q)$, Compos. Math. (\textbf{129}), no. 1,  1–-45, 2001. 

\bibitem{Vg} Vogan, D. A.: Representations of real reductive Lie groups,
   Progr. Math., vol. 15, Birkh\"{a}user, Boston, Mass., 1981.

\bibitem{Vo89} Vogan, D. A.: Associated varieties and unipotent representations,
 Harmonic analysis on reductive groups, edited by W. Barker and P. Sally, Progr. Math., vol. 101, 315--388, Birkh\"{a}user, Boston-Basel-Berlin, 1991. 

\bibitem{Wa2} Wallach, N. R.: Real reductive groups II, Academic Press Inc., 1992. 

\bibitem{WZ} Wong, D. and Zhang, H.: The genuine unitary dual of $\mathrm{Spin}(2n,\mathbb{C})$, arXiv: 2302.10803, 2023. 

\bibitem{Zu} Zuckerman, G.: Tensor products of finite and infinite dimensional representations of semisimple Lie groups, Ann. of Math. (\textbf{106}), 295--309, 1977. 

 
\end{thebibliography}
\end{document}